\documentclass[10pt]{amsart}
    \usepackage{amsmath}
    \usepackage{amssymb,amsfonts}
    \usepackage[all,arc]{xy}
    \usepackage{enumerate}
    \usepackage{mathrsfs}
    \usepackage{amsthm}
    \usepackage{bm} 
    \usepackage{color}
    \usepackage{graphicx}
    \usepackage{graphics}
    \usepackage{enumitem}
    \usepackage{extarrows}
    \usepackage{verbatim}
    \usepackage{dsfont}
    \usepackage{tikz-cd}
    \usepackage{tensor}

    \newtheorem{thm}{Theorem}[section]
    \newtheorem{cor}[thm]{Corollary}
    \newtheorem{prop}[thm]{Proposition}
    \newtheorem{lem}[thm]{Lemma}

    \theoremstyle{definition}
    \newtheorem{defn}[thm]{Definition}

    \newtheorem{exmp}[thm]{Example}

    \theoremstyle{remark}
    
    \newtheorem{rem}[thm]{Remark}

    \newcommand{\Z}{\mathbb{Z}}
    
    \newcommand{\C}{\mathbb{C}}

    \newcommand{\cO}{\mathcal{O}}

    \newcommand{\modd}{\mathrm{mod{-}}}
    
    \newcommand{\coh}{\mathrm{coh}}

    \newcommand{\Proj}{\mathrm{Proj}}

    \newcommand{\Hom}{\mathrm{Hom}}
    \newcommand{\Ext}{\mathrm{Ext}}
    
    \makeatletter
    \newcommand*{\rom}[1]{\expandafter\@slowromancap\romannumeral #1@}
    
    \makeatletter
    \let\c@equation\c@thm
    \makeatother
    \numberwithin{equation}{section}
    
\begin{document}

    \title{Moduli of Quiver Representations for Exceptional Collections on Surfaces \rom{1}}

   \author{Xuqiang QIN}
\address{Department of Mathematics, Indiana University, 831 E. Third St., Bloomington, IN 47405, USA}
\email{qinx@iu.edu}

\author{Shizhuo ZHANG}
\address{Department of Mathematics, Indiana University, 831 E. Third St., Bloomington, IN 47405, USA}
\email{zhang398@iu.edu}

    \begin{abstract}
    Suppose $S$ is a smooth projective surface over an algebraically closed field $k$, $\mathcal{L}=\{L_1,\ldots,L_n\}$ is a full strong exceptional collection of line bundles on $S$. Let $Q$ be the quiver associated to this collection. One might hope that $S$ is the moduli space of representations of $Q$ with dimension vector $(1,\ldots,1)$ for a suitably chosen stability condition $\theta$: $S\cong M_\theta$. In this paper, we show that this is the case for del Pezzo surfaces. Furthermore, we show the blow-up at a point can be recovered from an augmentation of exceptional collections (in the sense of L. Hille and M.Perling) via morphism between moduli of quiver representations. 
    \end{abstract}
    
    \maketitle
    
    \section{Introduction}
    Let $X$ be a smooth projective variety over an algebraically closed field $\mathbf{k}$ of characteristic $0$. Recall that objects $\mathcal{E}_1,\ldots,\mathcal{E}_n$ in the bounded derived category of coherent sheaves $\mathcal{D}^b(\coh(X))$ forms a full exceptional collection if 
    \begin{enumerate}
          \item $\Hom(\mathcal{E}_i,\mathcal{E}_i[m])=\mathbf{k}$ if $m=0$ and is $0$ otherwise;
          \item $\Hom(\mathcal{E}_i,\mathcal{E}_j[m])=0$ for all $m\in\Z$ if $j<i$;
          \item The smallest triangulated subcategory of $\mathcal{D}^b(\coh(X))$ containing $\mathcal{E}_1\ldots,\mathcal{E}_n$ is itself.
    \end{enumerate}
    
    An exceptional collection is strong if in addition:
    $\Hom(\mathcal{E}_i,\mathcal{E}_j[m])=0$ for all $i,j$ if $m\neq 0$. \
    
   In this paper we are only concerned with the case when $X$ is a smooth projective surface, the objects $\mathcal{E}_i$ are line bundles and the exceptional collection is strong.
   In this situation, we can consider the finite dimensional associative algebra
   \begin{equation*}
          \mathcal{A}=\mathrm{End}(\oplus_{i=1}^n \mathcal{E}_i)
    \end{equation*}
    it is well know that there is an exact equivalence of derived categories
    \begin{align*}
        \mathbf{R}\mathrm{Hom}(\oplus_{i=1}^n \mathcal{E}_i,-):\mathcal{D}^b(\coh(X))\to D^b(\modd \mathcal{A})
    \end{align*}
    whose inverse is given by 
    \begin{align*}
        (-)\otimes^L(\oplus_{i=1}^n \mathcal{E}_i):D^b(\modd \mathcal{A})\to {D}^b(\coh(X))
    \end{align*}
    This gives a non-commutative interpretation of the derived category of $X$. We note that when we input the structure sheaf $\cO_x$ of a close point $x\in X$ into the first functor, we obtain:
    \begin{align*}
        \mathbf{R}\mathrm{Hom}(\oplus_{i=1}^n \mathcal{E}_i,\cO_x)&=\mathrm{Hom}(\oplus_{i=1}^n \mathcal{E}_i,\cO_x)\\
        &=\oplus_{i=1}^n (\mathcal{E}_i^{\vee})_x
    \end{align*}
    Note since $\mathcal{A}^{op}=\mathrm{End}\big(\oplus_{i=1}^n (\mathcal{E}_i^{\vee})\big)$ is a finite dimensional algebra, there exist a bound quiver $(Q,I)$ such that giving an $\mathcal{A}^{op}$-module is equivalent to giving a representation of $(Q,I)$.\\ 
    
    King\cite{King} proved that when restricted by a stability condition $\theta$, the moduli space $M_{\theta}(\alpha)$ of semistable representations of $(Q,I)$ with any dimension vector $\alpha$ is a projective scheme. In our situation, for each point $x\in X$, one can associate the representation $\oplus_{i=1}^n (\mathcal{E}_i^{\vee})_x$ of $\mathcal{A}^{op}$, which has dimension vector $(1,\ldots,1)$. This provides a tautological map \begin{align*}
        T_0:X\to \mathcal{R}ep_{(1,\ldots,1)}(Q)
    \end{align*}  
    to the moduli stack of representation of $(Q,I)$ with dimension vector $(1,\ldots,1)$. Following \cite{BP}, we note $T_0$
    induce a tautological rational map $T:X\dashrightarrow M_\theta$  where $M_\theta$ is the moduli space of semistable representation of $(Q,I)$ with dimension vector $(1,\ldots,1)$. We note $M_\theta$ is a projective scheme and 
    \begin{align*}
        T(x)&=\mathrm{Hom}(\oplus_{i=1}^n \mathcal{E}_i,\cO_x)\\
        &=\oplus_{i=1}^n (\mathcal{E}_i^{\vee})_x
    \end{align*}
    if $x\in X$ is in the domain of $T$. It is natural to wonder when $T$ is a morphism and the relation between $M_{\theta}$ and $X$ for various $\theta$. In particular, one can ask if it is possible to choose $\theta$ so that $T$ is defined everywhere on $X$ and is an isomorphism.\\
    
    To study this question, it is necessary to understand the full strong exceptional collection of line bundles on surfaces. It is conjectured that the derived category of a smooth projective surface admits a full exceptional collection of line bundles if and only if the surface is rational. In the paper \cite{HP}, L.Hille and M.Perling introduced an operation called augmentation on exceptional collection of line bundles which we recall next:
    
    Let $S_0$ be a smooth surface and let $\pi\colon S \to S_0$ be a blowup of a point $p\in S_0$ with the corresponding $(-1)$-curve $E\subset S$. Let 
\begin{equation} \label{below}\{\cO _{S_0}(D_1),...,\cO _{S_0}(D_n)\}
\end{equation}
be a collection of line bundles on $S_0$. For some $1\le i\le n$ consider the 
collection 
\begin{equation}\label{stand-augmen}
\{\cO _S(\pi ^*D_1+E),...,\cO _S(\pi ^*D_{i-1}+E),\cO _S(\pi ^*D_i),\cO_S (\pi^*D_i+E),
\cO_SX(\pi ^*D_{i+1}),...,\cO_S(\pi ^*D_n)\}
\end{equation} 
The collection (\ref{stand-augmen}) 
is called an \emph{augmentation} of the collection
(\ref{below}). If collection (\ref{stand-augmen}) is a full strong exceptional collection on $S$, then 
(\ref{below}) is also such. The converse is not true in general but is often true under some assumptions on collection (\ref{below}), see \cite[Prop. 2.18]{EL}.  A collection is called a \emph{standard augmentation} if it is obtained by a series of augmentations from a collection on $\mathbb{P}^2$ or a Hirzebruch surface. 

The present paper focuses on del Pezzo surfaces. Elagin and Lunts\cite{EL} proved that every full strong exceptional collection of line bundles on a del Pezzo surface is a standard augmentation. 

The main result of this paper is the following theorem:

\begin{thm}[Main Theorem]\label{main}
Let $S$ be a del Pezzo surface of degree $\geq 3$ not isomorphic to $\mathbb{P}^1\times\mathbb{P}^1$. Let $\{\cO_S(D_1),\ldots,\cO_S(D_n)\}$ be any full strong exceptional collection of line bundles on $S$ so that
\begin{equation}\label{technical}
    -K_S-(D_n-D_1)=H-E_1-\ldots-E_k
\end{equation}
where $E_i$ are exceptional divisors and $0\leq k\leq 3$. Let $\mathcal{A}=\mathrm{End}\Big(\oplus_{i=1}^n\mathcal{O}_S(D_i)\Big)$ be the endomorphism algebra. Let $(Q,I)$ be a bound quiver for $\mathcal{A}^{op}$. Then one can choose (many) stability conditions $\theta$ such that the coarse moduli space $M_\theta$ of $\theta$-semistable representations of $(Q,I)$ with dimension vector $(1,\ldots,1)$ is a fine moduli space, and the tautological rational map:
\begin{equation*}
    T:S\dashrightarrow M_\theta
\end{equation*}
is an isomorphism.
\end{thm}

\begin{rem}
   The techniques to prove this result apply to more general rational surfaces (weak del Pezzo surfaces in particular), we only prove the del Pezzo cases for the sake of presentation. The curious readers are referred to the extended version of the present paper\cite{QZ}. The techniques also apply to some varieties in higher dimension, see \cite{Q}.
\end{rem}

    The proof of the main theorem uses induction on standard augmentations. Let $S_0$ be a surface and $\mathcal{L}_0$ be a full strong exceptional collection on $S_0$ satisfying the assumptions in the Main Theorem. Assume $S$ is the blow up of $S_0$ at a point and $\mathcal{L}$ is a full strong exceptional collection obtained by standard augmentation on $\mathcal{L}_0$ and satisfy (\ref{technical}). Under the induction hypothesis that there exists $\theta_0$ so that $T_0: S_0 \to M_{\theta_0}$ is an isomorphism, we first show how to find a suitable stability condition $\theta$ for the bound quiver of $\mathcal{L}$. Second, we construct a morphism $f:M_\theta\to M_{\theta_0}$ between moduli spaces with our choice of stability condition. Then we show that the induced morphism between moduli spaces is actually the blow up at a point and the diagram
     \[ \begin{tikzcd}
      S \arrow[r,"T"] \arrow{d}{\pi} & M_\theta \arrow{d}{f} \\%
      S_0 \arrow{r}{T_0}& M_{\theta_0}
      \end{tikzcd}
      \]
     is commutative. As a result, we have
     \begin{thm}
        With the above notation, one can always construct a morphism:
        \begin{align*}
            f:M_\theta\to M_{\theta_0}
        \end{align*}
        so that it realizes the blowing up  $\pi:S\to S_0$ as a morphism between moduli spaces of stable quiver representations.
   \end{thm}
  \begin{rem}
    This result shows that we can recover the blow up morphism at a point from augmentations of exceptional collections associated to it.
  \end{rem}
  
  By also considering the case $\mathbb{P}^1\times\mathbb{P}^1$ separately, we obtain:
   \begin{cor}\label{del pezzo}
    Every smooth del Pezzo surface $S$ of degree $\geq 3$ is a moduli space of stable representations (with respect to a suitable stability condition $\theta$) for a bound quiver $(Q,I)$ of a full strong exceptional collection of line bundles on $S$. Indeed, the tautological rational map:
    \begin{align*}
        T:S\dashrightarrow M_{\theta}
    \end{align*}
    is an isomorphism. \
\end{cor}

\begin{rem}
In \cite{King1}, A.King showed that every Hirzebruch surface is a moduli space of stable representation of quiver for full strong exceptional collection of line bundles on it.
\end{rem}

    There are many related results in the direction of realizing varieties as moduli spaces. Bergman-Proudfoot\cite{BP} proved that if a variety $X$ admits a full strong exceptional collection and the character is great (a base-point-freeness condition), then $X$ is a connected component of $M_\theta$. Craw-Smith\cite{CS} showed that if $X$ is a projective toric variety, then one can choose a collection of globally generated line bundles $\mathcal{L}_1, \dots,\mathcal{L}_n$ and a canonical choice of weight $\theta$, so that the moduli space $M_\theta$ is isomorphic to the variety $X$. Craw-Winn\cite{CW} extended this result to all Mori Dream spaces. Our results are different from theirs in the sense that our collection of line bundles is not required to be globally generated and our method applies to many rational surfaces with $K_X^2<0$ although we only treat del Pezzo surfaces in this paper. Furthermore, We actually construct a birational morphism between moduli space of representations of quiver coming from different strong exceptional collections of line bundles and show it is a blow up map. It is interesting to note that  Y.Toda proved a similar statement in \cite[Corollary 1.4]{To}. But there, he realized any smooth surface as moduli space of Bridgeland stable objects in derived category instead of quiver moduli. The behaviour of the space of Bridgeland stability conditions induced by the augmentations of exceptional collection of line bundles will be addressed in our future work.

    This paper is organized as follows: In the next section we recall some definitions and preliminary results that will be useful for later sections. In Section 3, we prove a few results on how standard augmentation affects the quiver of sections. In Section 4, we discuss the case of $S=\mathbb{P}^2$, which is the base case of induction. The last three sections are devoted to the proof of the main theorem.
    \subsection*{Notations and Conventions}
\begin{itemize}
    \item If $\vec{v}$ is a vector of size $n$, then we use $v_i$ to denote the $i$-th entry of the vector for $1\leq i\leq n$.
          
    \item If $\mathcal{L}$ is a line bundle on a scheme $S$ and $s\in H^0(X,\mathcal{L})$ is a section, we use $div(s)$ to denote the divisor of zeros of $s$.
\end{itemize}
    
\section{Preliminaries}
    
\subsection{Surfaces and divisors}
By a surface we mean a nonsingular projective integral separated scheme of dimension $2$ over $\C$. A surface is rational if it is birational to the projective plane. Let $S$ be a rational smooth projective surface over an algebraically closed field $k$ of characteristic zero.  Let $K_S$ be the canonical divisor on $S$.

Let $d=K_S^2$ be the degree of $S$, further we always assume that $d>0$. 
The Picard group $Pic(S)$ of $S$ is a finitely generated abelian group of rank $10-d$.
It is equipped with the intersection form $(D_1,D_2)\mapsto D_1\cdot D_2$ which has signature $(1,9-d)$.
For a divisor $D$ on $S$, we will use the following shorthand notations:
$$H^i(D):=H^i(S,\cO_S(D)),\quad h^i(D)=\dim H^i(D),\quad \chi(D)=h^0(D)-h^1(D)+h^2(D).$$
By the Riemann-Roch formula, one has
$$\chi(D)=1+\frac{D\cdot (D-K_S)}2.$$
 
The following notions are introduced by Hille and Perling in \cite[Definition 3.1]{HP}. 
\begin{defn}
A divisor $D$ on $S$ is \emph{numerically left-orthogonal} if $\chi(-D)=0$ (or equivalently $D^2+D\cdot K_S=-2$).\ 

A divisor $D$ on $S$ is \emph{left-orthogonal} (or briefly \emph{lo}) if $h^i(-D)=0$ for all $i$.\ 

A divisor $D$ on $S$ is \emph{strong left-orthogonal} (or briefly \emph{slo}) if $h^i(-D)=0$ for all $i$ and $h^i(D)=0$ for $i\ne 0$.
\end{defn}

The following propositions are easy but useful consequences of Riemann-Roch formula, see~\cite[Lemma 3.3]{HP} or \cite[Lemma 2.10, Lemma 2.11]{EL}.
\begin{prop}
Let $D$ be a numerically left-orthogonal divisor on $S$. Then 
$$\chi(D)=D^2+2=-D\cdot K_S.$$
\end{prop}
\begin{prop}
\label{hom}
Suppose $D_1,D_2$ are numerically left-orthogonal divisors on $S$. 
Then $D_1+D_2$ is numerically left-orthogonal if and only if $D_1D_2=1$. If that is the case, then 
$$\chi(D_1+D_2)=\chi(D_1)+\chi(D_2)\qquad \text{and}\qquad (D_1+D_2)^2=D_1^2+D_2^2+2.$$
\end{prop}
    
\begin{defn}
    A surface $S$ is del Pezzo if $-K_S$ is ample.
\end{defn}
Del Pezzo surfaces are classified up by their degree $d=K_S^2$. If $d=8$, then the surface is isomorphic to $\mathbb{P}^1\times\mathbb{P}^1$ or the blow up of $\mathbb{P}^2$ at a point. Otherwise, the surface is obtained by blowing up $(9-d)$ points at $\mathbb{P}^2$ with no three collinear, no six on a conic, and no eight of them on a cubic having a node at one of them. We call these points in general position. 

Let $p:S_0\to \mathbb{P}^2$ be a surface obtained from blowing up $n (n\leq 6)$ points on projective plane in general position. Let $\pi:S\to S_0$ be a surface obtained from blowing up one more point $P$ on the plane so that all $n+1$ points are in general position. Let $H$ be the class of hyperplane on $S$ or $S_0$. We make some obvious observations:
    \begin{lem}\label{bpf1}
      The linear system $|H|$ and $|H-E_1|$ on $S_0$ are both base point free for any exceptional divisor $E_1$ on $S_0$. In particular
        \begin{align*}
            h^0(S,H-E)=h^0(S,H)-1=2
        \end{align*}
        \begin{align*}
            h^0(S,(H-E-E_1)=h^0(S,H-E_1)-1=1
        \end{align*}
    \end{lem}

     \begin{lem}\label{surj23}
       Let $F$ be an exceptional divisor, then the natural map
       \begin{align*}
        H^0(S,H)\otimes H^0(S,H-F)\to H^0(S,2H-F)
    \end{align*}
    is surjective.
     \end{lem}
    \begin{proof}
      Note $h^0(S,H-F)=2$. Pick two linear functions $f,g$ such that 
      $$H^0(S,H-F)=span  (f,g)$$
      We can find another linear function $h$ so that 
      $$H^0(S,H)=span(f,g,h)$$
      then $H^0(S,H)\otimes H^0(S,H-F)=span(f^2,fg,g^2,fh,gh)$ which is a vector space of dimension five. Since $h^0(2H-F)=5$, we see the natural inclusion is surjective.
    \end{proof}
    \begin{lem}\label{surj22}
    Let $E_1$,$E_2$ be two exceptional divisor on $S$, then the natural map
    \begin{align*}
        H^0(S,H-E_1)\otimes H^0(S,H-E_2)\to H^0(S,2H-E_1-E_2)
    \end{align*}
    is surjective.
    \end{lem}
    \begin{proof}
      By Corollary \ref{bpf1}, $h^0(S,H-E_1-E_2)=1$, let the corresponding linear function be $f$. Then we can choose linear function $g,h$ such that $$H^0(S,H-E_1)=span  (f,g)$$$$H^0(S,H-E_2)=span(f,h)$$
      By our assumption, $g,h$ are linearly independent, then the image of $H^0(S,H-E_1)\otimes H^0(S,H-E_2)$ contains $f^2,fg,fh,gh$, which span a 4 dimensional vector space. Since $h^0(S,2H-E_1-E_2)=4$, we see the natural map is surjective.
    \end{proof}

\subsection{Exceptional Toric systems and standard augmentations}
We recall the notion of a \emph{toric system}, introduced by Hille and Perling in \cite{HP}.
For a sequence $\{\cO_S(D_1),\ldots,\cO_S(D_n)\}$ of line bundles one can consider 
the infinite sequence (called a \emph{helix}) $\{\cO_S(D_i)\}, i\in \Z$, defined by the rule
$D_{k+n}=D_k-K_S$. From Serre duality it follows that the collection $\{\cO_S(D_1),\ldots,\cO_S(D_n)\}$ is exceptional (resp. numerically exceptional) if and only if any collection of the form $\{\cO_S(D_{k+1}),\ldots,\cO_S(D_{k+n})\}$ is exceptional (resp. numerically exceptional). One can consider the $n$-periodic sequence $A_k=D_{k+1}-D_k$ of divisors on $S$. Following Hille and Perling, we will consider the finite sequence $\{A_1,\ldots,A_n\}$ with the cyclic order and will treat the index $k$ in $A_k$ as a residue modulo $n$.  Vice versa, for any sequence $\{A_1,\ldots,A_n\}$ one can construct the infinite sequence $\{\cO_S(D_i)\}, i\in \Z$, with the property $D_{k+1}-D_k=A_{k\mod n}$.

\begin{defn}[See {\cite[Definitions 3.4 and 2.6]{HP}}]
\label{def_ts}
A sequence $\{A_1,\ldots,A_n\}$ in $Pic(S)$ is called a toric system if $n=\text{rank } K_0(S)$ and the following conditions are satisfied (where indexes are treated modulo $n$):
\begin{itemize}
\item $A_i\cdot A_{i+1}=1$;
\item $A_i\cdot A_{j}=0$ if $j\ne i,i\pm 1$;
\item $A_1+\ldots+A_n=-K_S$.
\end{itemize}
\end{defn}

Note that a cyclic shift $\{A_k,A_{k+1},\ldots,A_n,A_1,\ldots,A_{k-1}\}$ of a toric system $\{A_1,\ldots,A_n\}$  is also a toric system. Also, note that by our definition any toric system has maximal length.

\begin{exmp}
$\{H,H,H\}$ is a toric system on $\mathbb{P}^2$. $\{H-E_1,E_1,H-E_1\}$ is a toric system on $\mathbb{F}_1\simeq Bl_p\mathbb{P}^2$. 
\end{exmp}

\begin{defn}
Given an exceptional collection of line bundles $\{\mathcal{O}_S(D_1),\ldots,\mathcal{O}_S(D_n)\}$, the corresponding toric system is defined to be a system of divisors $\{A_1,\ldots,A_n\}$ where
    \begin{align*}
        A_i=
        \begin{cases}
        D_{i+1}-D_i &\text{if $i=1,2,\ldots,n-1$}\\
        -K_S-(D_n-D_1) &\text{if $i=n$.}
        \end{cases}
    \end{align*}
\end{defn}

For the future use we make the following remark, see the proof of Proposition 2.7 in~\cite{HP}.
\begin{rem}
\label{remark_genpic}
For any toric system $\{A_1,\ldots,A_n\}$ the elements $A_1,\ldots,A_n$ generate $Pic(S)$ as abelian group.
\end{rem}

A toric system $\{A_1,\ldots,A_n\}$ is called \emph{exceptional} (resp. \emph{strong exceptional}) if the corresponding  collection $\{\cO_S(D_1),\ldots,\cO_S(D_n)\}$ is  exceptional (resp. strong exceptional). 


A toric system $\{A_1,\ldots,A_n\}$ is called \emph{cyclic strong exceptional} if the collection 
$$\{\cO_S(D_{k+1}),\ldots,\cO_S(D_{k+n})\}$$ 
is strong exceptional for any $k\in \Z$. Equivalently: if all cyclic shifts $\{A_k,A_{k+1},\ldots,A_n,A_1,\ldots,A_{k-1}\}$
are strong exceptional.

Notation: for a toric system $\{A_1,\ldots,A_n\}$ denote 
$$A_{k,k+1,\ldots,l}=A_k+A_{k+1}+\ldots+A_l.$$
We allow $k>l$ and treat 
$[k,k+1,\ldots,l]\subset [1,\ldots, n]$ as a cyclic segment.
Note that $A_{k\ldots l}$ is a numerically left-orthogonal divisor with 
\begin{equation}
\label{eq_Akl2}
A_{k\ldots l}^2+2=\sum_{i=k}^l(A_i^2+2).
\end{equation}

\begin{rem}
If for a toric system $A$ one has $A_i^2\ge -2$ for all $i$, then one has $A_{k,\ldots,l}^2\ge -2$ for any cyclic segment $[k,k+1,\ldots,l]\subset [1,\ldots, n]$. 
\end{rem}

\begin{rem}
If $A_{k\ldots l}$ is a strong left-orthogonal divisor, then $h^0(A_{k\ldots l})=A_{k\ldots l}^2+2=\sum_{i=k}^l(A_i^2+2)=\sum_{i=k}^lA_i^2+2(l-k+1)$.
\end{rem}

The following theorem is proved in \cite{EXZ}
\begin{thm}
\label{theorem_checkonlyminustwo}
Let $A=\{A_1,\ldots,A_n\}$ be a toric system on a del Pezzo surface $S$. Suppose $A^2_i\ge -2$ for  all $i$
Then $A$ is cyclic strong exceptional toric system.
\end{thm}

\subsubsection{Augmentations}
Following Hille and Perling \cite{HP}, we define augmentations. They provide a wide class of explicitly constructed toric systems.

Let $A'=\{A'_1,\ldots,A'_n\}$ be a toric system on a surface $S'$, and let $p\colon S\to S'$ be the blow up of a point with the exceptional divisor $E\subset S$. Denote $A_i=p^*A'_i$. Then one has the following toric systems on $S$:
\begin{align}
{\rm augm}_{p,1}(A')=&\{E,A_1-E,A_2,\ldots, A_{n-1},A_n-E\};\\
{\rm augm}_{p,m}(A')=&\{A_1,\ldots,A_{m-2}, A_{m-1}-E,E,A_{m}-E,A_{m+1},\ldots,A_n\}\quad\text{for}\quad 2\le m\le n;\\
\label{eq_augmentation3}
{\rm augm}_{p,n+1}(A')=&\{A_1-E,A_2,\ldots, A_{n-1},A_n-E,E\}.
\end{align}
Toric systems ${\rm augm}_{p,m}(A')$ ($1\le m\le n+1$) are called   \emph{elementary augmentations} of toric system  $\{A'_1,\ldots,A'_n\}$.

\begin{prop}[{See \cite[Proposition 3.3]{EL}}]
\label{prop_elemaugm}
In the above notation, let $A$ be a toric system on $S$ such that $A_m=E$ for some $m$. Then $A={\rm augm}_{p,m}(A')$ for some toric system $A'$ on $S'$.
\end{prop}

\begin{defn}
\label{def_augm1}
A toric system $A$ on $S$ is called a \emph{standard augmentation} if $S$ is a Hirzebruch surface or $A$ is an elementary augmentation of some standard augmentation. Equivalently: $A$ is a standard augmentation if there exists a chain of blow-ups
$$S=S_n\xrightarrow{p_n} S_{n-1}\to \ldots S_1\xrightarrow{p_1} S_0$$
where $S_0$ is a Hirzebruch surface and 
$$A={\rm augm}_{p_n,k_n}({\rm augm}_{p_{n-1},k_{n-1}}(\ldots {\rm augm}_{p_1,k_1}(A')\ldots ))$$
for some $k_1,\ldots,k_n$ and a toric system $A'$ on $S_0$. In this case we will say that $A$ is a \emph{standard augmentation along the chain $p_1,\ldots,p_n$}.
\end{defn}

\begin{rem}
To be more accurate, one  should add that (the unique) toric system $\{H,H,H\}$ on $\mathbb{P}^2$ is also considered as a standard augmentation.
\end{rem}

\begin{prop}\cite[Proposition 2.21]{EL}

\label{prop_augmexc}
Let $A={\rm augm}_k(A')$. Then
\begin{enumerate}
\item [(1)]$A$ is exceptional if and only if $A'$ is exceptional;
\item [(2)]if $A$ is strong exceptional then $A'$ is strong exceptional;
\item [(3)]if $A$ is cyclic strong exceptional then $A'$ is cyclic strong exceptional.
\end{enumerate}
\end{prop}

On del Pezzo surfaces, we have the following result:
\begin{thm}\cite[Theorem 1.3]{EL}
Any full exceptional collection of line bundles on a smooth del Pezzo surface is a standard augmentation
\end{thm}
In light of this theorem, if $S$ is a del Pezzo surface of degree $\geq 3$ not isomorphic to $\mathbb{P}^1\times\mathbb{P}^1$ and $\{\cO_S(D_1),\ldots,\cO_S(D_n)\}$ is a full strong exceptional collection of line bundles on $S$, then the collection comes from performing multiple steps of elementary augmentation on $\{\cO,\cO(H),\cO(2H)\}$ on $\mathbb{P}^2$ (see also Table 3 in \cite{EXZ}). The condition (\ref{technical}) then is the same as requiring we never perform elementary augmentation of the last type (\ref{eq_augmentation3}).

\subsection{Quivers and quiver representations} A quiver $Q$ is given by two sets $Q_{vx}$ and $Q_{ar}$, where the first set is the set of vertices and the second is the set of arrows, along with two functions $s,t:Q_{ar}\to Q_{vx}$ specifying the source and target of an arrow. The path algebra $\mathbf{k}Q$ is the associative $\mathbf{k}$-algebra whose underlying vector space has a basis consists of elements of $Q_{ar}$. The product of two basis elements is defined by concatenation of paths if possible, otherwise $0$. The product of two general elements is defined by extending the above linearly.\
    A bound quivers is a pair $(Q,I)$. Here $Q$ is a quiver and $I$ is a two sided ideal of $\mathbf{k}Q$ generated by elements of the form $\sum_{i=1}^nk_ip_i$, where $k_i\in \mathbf{k}^*$ and $p_i$ are paths with same heads and same tails for $i\in\{1,\ldots,n\}$. We simply use $Q$ to denote this pair when the existence of $I$ is understood.\\ 
    
    Let $Q$ be a quiver. A quiver representation $R=(R_v,r_a)$ consists of a vector space $R_v$ for each $v\in Q_{vx}$ and a morphism of vector spaces $r_a:R_{s(a)}\to R_{t(a)}$ for each $a\in Q_{ar}$. For a bound quivers $(Q,I)$, a representation  $R=(R_v,r_a)$ is same as above, with the additional condition that 
    \begin{align*}
        \sum_{i=1}^nk_ir_{p_i}=0
    \end{align*} if $\sum_{i=1}^nk_ip_i$ is a generator of $I$. A subrepresentation of $R$ is a pair $R'=(R_v',r_a')$ where $R_v'$ is a subspace of $R_v$ for each $v\in Q_{vx}$ and $r'_a:R'_{s(a)}\to R'_{t(a)}$ a morphism of vector spaces for each $a\in Q_{ar}$ such that $$r'_a=r_a|_{R'_{s(a)}}$$ and 
    \begin{equation}\label{subset}
        r_a(R'_{s(a)})\subset R'_{t(a)}
    \end{equation}
    Thus we have the commutative diagram
    \[ \begin{tikzcd}
      R'_i \arrow[r,"r'_a"] \arrow{d}{\iota_i} & R'_j \arrow{d}{\iota_j} \\%
      R_i \arrow{r}{r_a}& R_j
      \end{tikzcd}
      \]
    for any arrow $a$ from $i$ to $j$.
    We use $R'\subset R$ to denote that $R'$ is a subrepresentation of $R$.

    If the vertices of a quiver has a natural ordering, as it will be the case when we discuss quiver of sections of an exceptional collection of line bundles, we define dimension vector $\vec{d}$ so $d_i$ is the dimension of the vector space $R_i$ at that vertex. We call the set of vertices where $R_v$ has positive dimension the support of $R$.\\
    
    In this paper, we are particularly interested in representations with dimension vector $\mathbf{1}=(1,\ldots,1)$. Notice when $R$ is a representation with dimension vector $\mathbf{1}$, and $R'\subset R$, all the inclusion maps $$\iota_k:R'_k\to R_k$$ are either zero map or identity. We prove the following easy lemma:
    \begin{lem}\label{subrep}
      Let $(Q,I)$ be a bound quivers whose vertices are label by $\{0,1,2,\ldots,n\}$ and $R$ be a representation of $Q$ with dimension vector $\mathbf{1}$.Then any subrepresentation $R'$ is determined by its dimension vector $\vec{d}$. Moreover, a vector $\vec{d}$ of size $n+1$ with entries $0$ and $1$ is the dimension vector of a subrepresentation of $R$ if and only if $r_a=0$ for all $a\in Q_{ar}$ with $d_{s(a)}=1$ and $d_{t(a)}=0$.
      \end{lem}
      \begin{proof}
      Since $\dim R_i=1$, its subspaces are determined by dimensions. Moreover, we see the morphism of subspaces $r'_a$ are restrictions of $r_a$, hence the dimension vector $\vec{d}$ determines $R'_i$ for all $i\in Q_{vx}$.\\
      Given any vector $\vec{d}$ as in the second part of the lemma, it is the dimension of a vector subspace if (\ref{subset}) is satisfied. Note (\ref{subset}) is always true unless for arrows with $d_{s(a)}=1$ and $d_{t(a)}=0$, in which case we must have $r_a=0$.
      \end{proof}
      \begin{rem}
      We will use Lemma \ref{subrep} to construct subrepresentations. We will do so by prescribing dimension of the subrepresentation at each vertex, and check the vanishing of arrows as in Lemma \ref{subrep} along the way, until we have the whole dimension vector, which determines the subrepresentation.
      \end{rem}
      \begin{rem}\label{comparison}
      If $R$, $T$ are both representations with dimensional vector $\mathbf{1}$ and $R$ has more arrows with value $0$, i.e. $t_a=0$ implies $r_a=0$ for $a\in Q_{ar}$, then $R$ has more subrepresentations than $T$, more precisely, if $T$ has a subrepresentation with dimension vector $\vec{d}$, then $S$ also has a subrepresentation with the same dimension vector.
      \end{rem}

    \subsection{Moduli space of semistable representations of a quiver}
    Given a bound quivers $(Q,I)$,a weight is an element $\theta\in\Z^N$ where $N=|Q_{vx}|$ such that $\sum_{i=1}^{N}\theta_i=0$. Let $\theta=(\theta_1,\ldots,\theta_{N})$ be a weight, we defined its toric form to be
    \begin{align*}
        (-\theta_1,-\theta_1-\theta_2,\ldots,-\theta_1-\theta_2-\ldots-\theta_{N-1})\in \Z^{n-1}
    \end{align*}
    It is an easy exercise to see that one can recover a weight from its toric form.
    \begin{defn}
      A weight is admissible if every entry of its toric form is a positive integer.
    \end{defn}
    For a weight $\theta$, the weight function is defined by by:
    \begin{equation*}
        \theta(S)=\sum_{i=1}^N d_i \theta_i
    \end{equation*}
    where $S$ is a representation of $Q$ and $d_i$ and $\theta_i$ are the $i$-th entries of $\vec{d}$ and $\theta$ respectively. We recall the definition of semi-stability:
    \begin{defn}
        A representation $R$ is $\theta$-semistabe if for any subrepresentation $R'\subset R$
    \begin{equation*}
        \theta(R')\geq0
    \end{equation*}
    $R$ is $\theta$-stable if all the above inequalities are strict.
    \end{defn}
    
    We restrict our attention to $R$ with dimension vector $\mathbf{1}$. Given a bound quivers $(Q,I)$, we can associate to it an affine shceme $\mathrm{Rep(Q)}$ called the representation scheme of $(Q,I)$. The coordinate ring of this affine shceme is the quotient of $\mathbf{k}[a\in Q_{ar}]$ by the ideal $J$ which is generated by generators $\sum_{i=1}^nk_ip_i$ of $I$ treated as  elements in the above polynomial ring. It is obvious from the definition that closed points of representation scheme are in 1-to-1 correspondence with representations of $Q$ with dimension vector $\mathbf{1}$. For a weight $\theta$, the set of $\theta$-semistable representations forms an open subscheme $\mathrm{Rep(Q)}^{SS}_\theta$ of $\mathrm{Rep(Q)}$, the set of $\theta$-stable representations forms an open subscheme $\mathrm{Rep(Q)}^{S}_\theta$ of $\mathrm{Rep(Q)}^{SS}_\theta$.
    
    The group $(\mathbf{k}^*)^{Q_{vx}}$ acts by incidence on $\mathrm{Rep(Q)}$, in other words, it acts by $(g\cdot a)=g_{t(a)}r_ag^{-1}_{s(a)}$. Apparently, the diagonal subgroup $\mathbf{k}_{\text{diag}}^*$ of $(\mathbf{k}^*)^{Q_{vx}}$ consisting of elements of the form $(k,k,\ldots,k)$ for $k\in\mathbf{k}^*$ acts trivially on $\mathrm{Rep(Q)}$. So it is natural to only consider the action of $\mathrm{PGL(\mathbf{1})}:=(\mathbf{k}^*)^{Q_{vx}}/\mathbf{k}_{\text{diag}}^*$. 
    \begin{defn}
      Two representations of dimension vector $\mathbf{1}$ are isomorphic if they are in the same orbit under the action of $\mathrm{PGL(\mathbf{1})}$.\
    \end{defn}

    Give a weight $\theta$, the moduli space of $\theta$-semistable representation with dimension vector $\mathbf{1}$ is the GIT quotient
    \begin{align*}
        M_\theta:&=\mathrm{Rep(Q)}//_\theta \mathrm{PGL(\mathbf{1})}
    \end{align*}
    We mention a few facts about $M_\theta$. For details, the readers are referred to \cite{King}.
    An equivalent definition of $M_{\theta}$ is to consider the graded ring
       $$ B_{\theta}= \bigoplus_{r\geq 0}B(r\theta)$$
    where $B(r\theta)$ is $r\theta$-semi-invariant functions in the coordinate ring of $\mathrm{Rep(Q)}$. Then the GIT quotient is defined as 
    \begin{equation*}
        M_\theta=\Proj(B_\theta)
    \end{equation*}
    From this definition, it is easy to see that $M_{\theta}$ is a reduced projective scheme. Note if all $\theta$-semistable representations are $\theta$-stable, i.e. $\mathrm{Rep(Q)}^{SS}_{\theta}=\mathrm{Rep(Q)}^{S}_{\theta}$, then $M_{\theta}$ is the fine moduli space of $\theta$-stable representations, in particular, the closed points of $M_{\theta}$ are in 1-to-1 correspondence with the isomorphism classes of $\theta$-stable representations.
    \
    We now give an easy criterion for obtaining fine moduli spaces as above.
    \begin{lem}\label{criterion}
    With the notions above, if for any proper nonempty subset $P$ of $Q_{vx}$, we have $\sum_{i\in P}\theta_i\neq 0$, then any semistable representation $R$ is in fact stable. In particular, $M_{\theta}$ is a fine moduli space. 
    \end{lem}
    \begin{proof}
      If $R$ is strictly semistable, then there exist a proper nonzero subrepresentation $R'$ such that $ \theta(R')=\sum_{i\in supp(R')}\theta_i =0$, but this cannot happen given the conditions in the statement.
    \end{proof}

    \subsection{Quivers of Sections} The main reference for this section is Craw-Smith\cite{CS} and Craw-Winn\cite{CW}. We mention that our indexing is different since we are concerned with the quiver with path algebra $\mathcal{A}^{op}$ instead of $\mathcal{A}$ as in the introduction.\\
    \noindent Let $\mathcal{L}=\{L_1,L_2,\ldots,L_n\}$ be a collection of line bundles on a projective variety $X$. For $1\leq i,j\leq n$, we call a section $s\in H^0(X,L_j^{\vee}\otimes L_i)$ irreducible if $s$ does not lie in the images of the multiplication map
    \begin{equation*}
        H^0(X,L_j^{\vee}\otimes L_k)\otimes_{\mathbf{k}}H^0(X,L_k^{\vee}\otimes L_i)\to H^0(X,L_j^{\vee}\otimes L_i)
    \end{equation*}
    for $k\neq i,j$. 
    \begin{defn}
      The quiver of sections of the collection $\mathcal{L}$ on $X$ is defined to be a quiver with vertex set $Q_{vx}=\{1,\ldots,n\}$ and where the arrows from $i$ to $j$ corresponds to a basis of irreducible sections of $H^0(X,L_{(n+1)-j}^{\vee}\otimes L_{(n+1)-i})$.
    \end{defn}
    \begin{rem}
        In Section 3, we will show there is a way of choosing the sections that the arrows represents which will aid our computation.
    \end{rem}
    We mention one of the basic properties of a quiver of sections.
    \begin{lem}{\cite{CW}}
    The quiver of sections $Q$ is connected, acyclic and $1\in Q_{vx}$ is the unique source.
    \end{lem}
    The quiver of sections only include information about the sections in $H^0(X,L_{(n+1)-j}^{\vee}\otimes L_{(n+1)-i})$, but left relations between them behind. We now define a two sided ideal
    \begin{defn}
    Let $I_\mathcal{L}$ be a two sided ideal in $\mathbf{k}Q$
    \begin{equation*}
        I_\mathcal{L}=\big(\sum_{k=1}^Na_kp_k|p_k \text{ are paths from }i \text{ to }j\ \text{ and }\sum_{k=1}^Na_kp_k \text{ represents } 0 \text{ in } H^0(X,L_{(n+1)-j}^{\vee}\otimes L_{(n+1)-i})\big)
    \end{equation*}
    We call the pair $(Q,I_\mathcal{L})$ the bound quiver of sections of the collection $\mathcal{L}$.
    \end{defn}
    \begin{prop}{\cite{CS}\cite{CW}}
      The quotient algebra $\mathbf{k}Q/I_\mathcal{L}$ is isomorphic to $\mathcal{A}^{op}=\mathrm{End}_{\cO_X}(\oplus_{i=1}^nL_i^\vee)$ and for $1\leq i,j\leq n$, we have $e_j(\mathbf{k}Q/I_\mathcal{L})e_i\cong H^0(X,L_{(n+1)-j}^{\vee}\otimes L_{(n+1)-i})$.
    \end{prop}
    
    \begin{rem}
    Suppose $\mathcal{L}=\{L_1,L_2,\ldots,L_n\}$ is an exceptional collection of line bundles on a projective variety $X$, and $\mathcal{TS}=\{A_1,\ldots,A_n\}$ is the corresponding toric system. If we define $\mathcal{TS}^{op}=\{B_1\ldots,B_n\}$ by
    \begin{align*}
        B_i=
        \begin{cases}
        A_{n-i}&\text{if $i=1,2,\ldots,n-1$}\\
        A_n &\text{if $i=n$.}
        \end{cases}
    \end{align*}
    Note $\mathcal{TS}^{op}$ is also an exceptional toric system (it remains strong if $\mathcal{TS}$ is), and 
    for $1\leq i<j\leq n$, we have $e_j(\mathbf{k}Q/I_\mathcal{L})e_i\cong H^0(X,B_i+\ldots+B_{j-1})$. In fact, $\mathcal{TS}^{op}$ is the corresponding toric system of $\{L_n^\vee,L_{n-1}^\vee,\dots,L_1^\vee\}$
    \end{rem}

    Given any weight $\theta$ for $(Q,I_\mathcal{L})$, we can consider the moduli space of semistable representations $M_\theta$. There is a tautological rational map
    \begin{equation*}
        T:X\dashrightarrow M_\theta
    \end{equation*}
    so that if $T$ is defined at $x$, then  $$T(x)=\bigoplus_{i=0}^n(L_i^\vee)_x$$ Moreover, $T$ is defined at $x$ if $\bigoplus_{i=0}^n(L_i^\vee)_x$ can be represented by a $\theta$-semistable representation.
    
    \section{Augmentation for quiver of sections}
    In this section we discuss properties of quivers of sections coming from full strong toric systems, and define the operation of augmentation between them.\\
    
    Let $S_0$ be a rational surface, $\mathcal{TS}_0=\{A_1,\ldots,A_n\}$ a full strong toric system on $S_0$. Let $\pi:S\to S_0$ be a blow up of $S_0$ at a single point $P$ with the exceptional curve $E$. Let $$\mathcal{TS}=\{\pi^*(A_1),\ldots,\pi^*(A_{n-1-k}),\pi^*(A_{n-k})-E,E,\pi^*(A_{n-k+1})-E,\pi^*(A_{n-k+2}),\ldots, \pi^*(A_n)\}$$ be a full strong exceptional toric system obtained from standard augmentation of $\mathcal{TS}_0$ at position $n-1-k$.\\
 For convenience, we define the opposite toric system, denoted by $\mathcal{TS}^{op}$ on a surface $S$:
    \begin{align*}
        \mathcal{TS}_0^{op}&=\{B_1,B_2,\ldots,B_{n-1},B_n\}\\
        &=\{A_{n-1},A_{n-2},\ldots,A_2,A_1,A_n\}
    \end{align*}
    then 
    \begin{align*}
        \mathcal{TS}^{op}&=\{\pi^*(B_1),\ldots,\pi^*(B_{k-2}),\pi^*(B_{k-1})-E,E,\pi^*(B_{k})-E,\pi^*(B_{k+1}),\ldots,\pi^*(B_n)\}\\
        &=\{\pi^*(A_{n-1}),\ldots,\pi^*(A_{n-k+2}),\pi^*(A_{n-k+1})-E,E,\pi^*(A_{n-k})-E,\pi^*(A_{n-k-1}),\ldots,\pi^*(A_1),\pi^*(A_n)\}
    \end{align*}
    \begin{rem}
    We label the vertices of $Q_0$ by $1,2,\ldots,n$. And label the vertices of $Q$ by $1,\ldots,k-1,k,k',k+1,\ldots,n$. In this way, we obtain an 1-to-1 correspondence between the vertices of $Q$ and $Q_0$ except at the place where the augmentation is performed.
    \end{rem}
    \begin{defn}
      We use $e$ to denote the unique arrow from $k$ to $k'$ corresponding to the unique section of the line bundle $E$.
    \end{defn}

    Note we have $\pi_*(\cO_S)=\cO_{S_0}$, hence by adjunction and projection formula:
    \begin{align}
          \Hom_{\cO_S}(\pi^*(L_1),\pi^*(L_2))
          &=\Hom_{\cO_{S_0}}(L_1,L_2)
    \end{align}
    This computation shows the section quiver of $\mathcal{TS}$ is same as the section quiver of $\mathcal{TS}_0$ when we look at parts which are either between vertex $1$ to $k-1$ or between vertex $k+1$ to $n$. To be more precise, there is a bijection between the arrows of $Q$ and $Q_0$ from $i$ to $j$, if $i\geq k+1$ or $j\leq k-1$, so that two corresponding arrows represent the same section under $(3.3)$.
    \begin{rem}
    From now on we omit $\pi^*$ when no confusion will be caused.
    \end{rem}
    We now analyze the arrows from vertex $i$ to vertex $k'$. where $i<k$.
    \begin{lem}\label{left}
      If $i<k$, we have a natural embedding
      \begin{equation*}
          H^0(S,B_i+\ldots +B_{k-1}-E)\hookrightarrow H^0(S,B_i+\ldots+B_{k-1})
      \end{equation*}
      and $H^0(S,B_i+\ldots+B_{k-1})$ is spanned by the image of $H^0(S,B_i+\ldots B_{k-1}-E)$ and an element which represents a divisor $f$ such that $div(f)-E$ is not effective.
      \end{lem}
      \begin{proof}
    Consider the standard short exact sequence
    \begin{equation*}
        0\rightarrow \cO_S(-E)\to \cO_S\to \cO_E\to0
    \end{equation*}
    Twist this exact sequence by $\cO_S(B_i+\ldots+B_{k-1})$, and note $E\cdot (B_i+\ldots+B_{k-1})=0$, we obtain
    \begin{equation*}
        0\rightarrow \cO_S(B_i+\ldots+B_{k-1}-E)\to \cO_S(B_i+\ldots+B_{k-1})\to \cO_E\to0
    \end{equation*}
    This induce long exact sequence
    \begin{align*}
        0\to H^0(S,B_i+\ldots+B_{k-1}-E)\to H^0(S,B_i+\ldots+B_{k-1})\to \\
        H^0(E,\cO_E)\to H^1(S,B_i+\ldots+B_{k-1}-E)
    \end{align*}
    This gives the inclusion in the statement.
    Using the formula in preliminary, we see 
    \begin{equation*}
          h^0(S,B_i+\ldots B_{k-1}-E)+1= h^0(S,B_i+\ldots+B_{k-1})
    \end{equation*}
    Hence there is an effective divisor $F$ linearly equivalent to $B_i+\ldots+B_{k-1}$ such that $F-E$ is not effective. The section corresponding to $F$ and $H^0(S,B_i+\ldots B_{k-1}-E)$ spans $H^0(S,B_i+\ldots+B_{k-1})$ by looking at the dimension.
    \end{proof}
    \begin{rem}
    The computation at the beginning of this section tells us that we can assume the sections of arrows from vertex $i$ to vertex $k'$ in $Q$ are in 1-to-1 correspondence with the sections of arrows from vertex $i$ to vertex $k$ in $Q_0$. This lemma tells us  we can require in addition that all but one arrows from vertex $i$ to vertex $k'$ in $Q$ comes from arrows from vertex $i$ to vertex $k$ composed with $e$, while there is a unique arrow from vertex $i$ to vertex $k'$ which is not a multiple of $e$.
    \end{rem}

    Using exact same argument, we obtain mirror result for arrows from vertex $k$ to vertex $j$ where $j>k$.
      \begin{lem}\label{right}
      If $j>k$, we have a natural embedding
      \begin{equation*}
          H^0(S,-E+B_{k}+\ldots B_{j-1})\hookrightarrow H^0(S,B_k+\ldots+B_{j-1})
      \end{equation*}
      and $H^0(S,B_{k+1}+\ldots+B_{j-1})$ is spanned by the image of $H^0(S,-E+B_{k+1}+\ldots B_{j-1})$ and an element $g$ such that $div(g)-E$ is not effective.
    \end{lem}
    \begin{defn}\label{arrow1}
      If $i<k$, we fix an arrow in $Q$ from $i$ to $k'$ which corresponds to a section
      \begin{align*}
          f\in H^0(S,B_i+\ldots+B_{k-1})
      \end{align*}
      such that $div(f)-E$ is not effective as in Lemma \ref{left}. Denote this arrow by $u_i$. We note $u_i$ is not the composition of an arrow from $i$ to $k$ with $e$. We use $w_i$ to denote the arrow in $Q_0$ from $i$ to $k$ that represents the same section as $f$ under the isomorphism
      \begin{equation*}
          H^0(S,\pi^*B_i+\ldots+(\pi^*B_{k-1}-E)+E)=H^0(S_0,B_i+\ldots+B_{k-1})
      \end{equation*}
      If $j>k$, we fix an arrow from $k$ to $j$ which corresponds to the section
      \begin{align*}
          g\in H^0(S,B_k+\ldots+B_{j-1})
      \end{align*}
      such that $div(g)-E$ is not effective as in Lemma \ref{right}. We denote this arrow by $u_j$. Note $u_j$ is not a composition $e$ with an arrow from $k'$ to $j$. We use $w_j$ denote the arrow in $Q_0$ from $i$ to $k$ that represents the same section as above.
    \end{defn}
    
    Finally we analyze arrows in $Q$ from vertex $i$ to vertex $j$ where $i<k<j$. Note $B_i+\ldots +(B_{k-1}-E)+E+(B_{k}-E)+\ldots+B_j=B_i+\ldots B_j-E$.
    
    \begin{lem}\label{cross arrow}
    If $i<k<j-1$, we have a natural embedding
      \begin{equation*}
          H^0(S,B_i+\ldots+ B_{j-1}-E)\hookrightarrow H^0(S,B_i+\ldots+B_{j-1})
      \end{equation*}
      and $H^0(S,B_i+\ldots+B_{j-1})$ is spanned by the image of $H^0(S,B_i+\ldots+ B_{j-1}-E)$ and the product of element $f$ in $H^0(S,B_i+\ldots+B_{k-1})\backslash H^0(S,B_i+\ldots B_{k-1}-E)$ with an element $g$ in $ H^0(S,B_{k+1}+\ldots+B_{j-1})\backslash H^0(S,-E+B_{k+1}+\ldots B_{j-1})$
      \end{lem}
      \begin{proof}
    Again the embedding is obtained as before. Using the formula in preliminary, we see
      \begin{equation*}
          h^0(S,B_i+\ldots+B_{j-1})= h^0(S,B_i+\ldots+B_{j-1}-E)+1
      \end{equation*}
      By Lemma 3.1 and 3.2, $f$ and $g$ as described above exists.
      Consider the product $H^0(S,B_i+\ldots+B_{k-1})\otimes H^0(S,B_{k}+\ldots+B_{j-1})\to H^0(S,B_i+\ldots+B_{j-1})$, then $fg\in H^0(S,B_i+\ldots+B_{j-1})$ with the additional property that $fg$ does not vanish at $P$, hence not in the image of  $H^0(S,B_i+\ldots+ B_{j-1}-E)$. Hence $fg$ and $H^0(S,B_i+\ldots+ B_{j-1}-E)$ span $H^0(S,B_i+\ldots+B_{j-1})$ by looking at the dimension.
      \end{proof}
    This lemma tells us that if $i<k<j$, then all but one arrows from vertex $i$ to vertex $j$ in $Q_0$ corresponds to arrows from vertex $i$ to $j$ in $Q$.
    
    \section{Example: $\mathbb{P}^2$}
    In this section we consider the case when $S=\mathbb{P}^2$. This case is on the one hand a motivating example and on the other hand the base case of induction by augmentation for the next section.\ 
    
    We use $H$ to denote a divisor corresponding to the hyperplane line bundle $\mathcal{O}_{\mathbb{P}^2}(1)$. We have the unique full cyclic strong toric system $\mathcal{TS}=\{H,H,H\}$ and $\mathcal{TS}^{op}=\{H,H,H\}$. The corresponding quiver of sections is
    \[
    \begin{tikzcd}
    1\arrow[r,shift left=0.5ex]
    \arrow[r,shift right=0.5ex]
    \arrow[r]
    & 2 \arrow[r,shift left=0.5ex]
       \arrow[r,shift right=0.5ex]
       \arrow[r,]
    &3
    \\
    \end{tikzcd}
    \]
    We call the three arrows from $0$ to $1$ by $x_1,y_1,z_1$, the three arrows from $1$ to $2$ by $x_2,y_2,z_2$ where $x_i,y_i,z_i$ represents the linear functions $x,y,z$ in $H^0(\mathbb{P}^2,H)$. Then we have relation:
    \begin{align*}
        y_2x_1&=x_2y_1\\
        y_2z_1&=z_2y_1\\
        z_2x_1&=z_1x_2
    \end{align*}
    We denote this quiver with relation by $Q$.\\
    \begin{thm}
      For $m,n\in \Z_{>0}$, $m\neq n$, if $\theta$ has toric form $(m,n)$, i.e. $\theta=(-m,m-n,n)$, then $$T:\mathbb{P}^2\dashrightarrow M_\theta$$ is an isomorphism. 
    \end{thm}
    
    \begin{proof}
     Take $s=[a:b:c]\in\mathbb{P}^2$, then $a,b,c$ cannot be $0$ simultaneously. Consider the isomorphism class $\cO(-H)_s\oplus \cO(-2H)_s\oplus \cO(-3H)_s$ of representations of $Q$. By change of bases for each of the three summands, we can find $R\in \cO(-H)_s\oplus \cO(-2H)_s\oplus \cO(-3H)_s$ so that
      \begin{align*}
          &r_{x_1}=r_{x_2}=a\\
          &r_{y_1}=r_{y_2}=b\\
          &r_{z_1}=r_{z_2}=c
      \end{align*}
    By Lemma \ref{subrep}, then only possible nontrivial proper subrepresentations  of $R$ has dimension vector $(0,1,1)$ or $(0,0,1)$, in both cases, if we use $S$ to denote the subrepresentation, we can see $\theta(S)>0$ using the fact that $m,n>0$. Thus $T$ is a morphism, i.e it is defined on all of $\mathbb{P}^2$.\\
    Now consider the natural map 
    \begin{align*}
        f_k:B(k\theta)\to H^0(\mathbb{P}^2,k(mH+nH))
    \end{align*}
    for $k\geq 0$. Note $B(k\theta)$ has a basis $\{x_1^{p_1}y_1^{q_1}z_1^{r_1}x_2^{p_2}y_2^{q_2}z_2^{r_2}\}_{p_1+q_1+r_1=km,p_2+q_2+r_2=kn}$ and this map maps $x_1^{p_1}y_1^{q_1}z_1^{r_1}x_2^{p_2}y_2^{q_2}z_2^{r_2}$ to $x^{p_1+p_2}y^{q_1+q_2}z^{r_1+r_2}$. Using the relations, one easily sees $f_n$ is an isomorphism for each $k\geq 0$ and $(f_k)_{k\geq0}$ gives a isomorphism of graded rings
    \begin{align*}
        f:\bigoplus_{k\geq0}B(k\theta)\to \bigoplus_{k\geq0}H^0(\mathbb{P}^2,k(mH+nH))
    \end{align*}
    Since the tautological map $T$ is obtained from $f$ by taking $\Proj$ of both sides, we see $T$ is an isomorphism.
    \end{proof}
    \begin{rem}
      It is clear that for $m,n\in\Z_{>0}$, $m\neq n$ and $\theta=(-m,m-n,n)$, $\theta$ is a weight that satisfies the condition in Lemma \ref{criterion}.
    \end{rem}

    \section{Choice of weight}
    The proof of Main Theorem is by induction. We adapt the notations at the beginning of Section 3. We assume furthermore that both $S,S_0$ are del Pezzo surfaces and the toric systems $\mathcal{TS}^{op}$ (hence $\mathcal{TS}^{op}_0$ too) satisfies the assumption $(1.4)$. We show that if there exist weight $\theta_0$ such that $T_0:S_0\dashrightarrow M_{\theta_0}$ is an isomorphism, then we can find a weight $\theta$ so that $T:S\dashrightarrow M_\theta$ is an isomorphism.\
    \begin{rem}\label{op}
      We observe that under the assumption (\ref{technical}), if
    \begin{align*}
        \mathcal{TS}_0^{op}&=\{B_1,B_2,\ldots,B_{n-1},B_n\}\\
        &=\{A_{n-1},A_{n-2},\ldots,A_2,A_1,A_n\}
    \end{align*}
    then 
    \begin{align*}
        \mathcal{TS}^{op}&=\{\pi^*(B_1),\ldots,\pi^*(B_{k-1})-E,E,\pi^*(B_{k})-E,\ldots,\pi^*(B_n)\}\\
        &=\{\pi^*(A_{n-1}),\ldots,\pi^*(A_{n-k+1})-E,E,\pi^*(A_{n-k})-E,\ldots,\pi^*(A_1),\pi^*(A_n)\}
    \end{align*}
    is a full strong toric system obtained by standard augmentation from $\mathcal{TS}^{op}$ at position $k$.
    \end{rem}

    We will always use weights that are admissible. Suppose $\theta_0$ is an admissible weight, let its toric form be $(b_1,b_2,\ldots,b_{n})$. We define $\theta$ in its toric form as $$(2b_1,2b_2,\ldots,2b_{k-1},(2b_{k-1}+2b_{k}-1),2b_{k},\ldots,2b_{n-1})$$, i.e. the exceptional divisor $E$ in the toric system $\mathcal{TS}^{op}$ is given weight $2b_{k-1}+2b_{k}-1$.
    Note if $k=1$, then the toric form of $\theta$ is $$(2b_1-1,2b_1,2b_2,\ldots,2b_{n-1})$$. If $k=n$, the toric form of $\theta$ is $$(2b_1,2b_2,\ldots,2b_{n-1},2b_{n-1}-1)$$. One can think of these two cases as a natural extension of the general case by thinking $b_k=0$ if $k<1$ or $k>n$. It is clear that $\theta$ defined in this way is also admissible.\\
      
    The motivation for such a choice is that we will choose in a fashion that $\sum_{i=1}^{n-1}b_iB_i$ is a very ample divisor on $S_0$, and our new weight will correspond to divisor $2\sum_{i=1}^{n-1}b_i\pi^*(B_i)-E$, which can be easily verified to be very ample on $S$ using Nakai-Moishezon criterion. We believe that any weight whose toric form gives a very ample divisor should cut out a moduli space isomorphic to $S$, although in this paper we only manage to prove some very special choices.\\
    The following proposition along with Lemma \ref{criterion} shows some nice properties our choice of $\theta$ enjoys:
    \begin{prop}\label{weight}
      With the above assumption,
      let $\theta_0$ whose toric form is $(b_1,\ldots,b_{n-1})$ be a character satisfying the conditions in Lemma \ref{criterion} on $S_0$, then $\theta$ defined as above, whose toric form is $$(2b_1,2b_2,\ldots,2b_{k-1},(2b_{k-1}+2b_{k}-1),2b_{k},\ldots,2b_{n-1})$$ also satisfies the conditions in Lemma \ref{criterion}
    \end{prop}
    \begin{proof}
      We use combinatorics to verify that $\theta$ satisfy the condition in the above lemma. We prove the proposition for $1<k<n$, the cases when $k=1$ or $k=n$ are handled exactly the same.
      Note $$\theta_0=(-b_1,b_1-b_2,\ldots,b_{n-2}-b_{n-1}, b_{n-1})$$ and $$\theta=\big(-2b_1,2(b_1-b_2),\ldots,2(b_{k-2}-b_{k-1}),1-2b_{k},2b_{k-1}-1,2(b_{k}-b_{k+1}),\ldots,2(b_{n-2}-b_{n-1}),2b_{n-1}\big)$$ With the assumption on $\theta$, any subset of entries of $\theta$ not involving $1-2b_{k}$ and $2b_{k-1}-1$ cannot sum up to 0. If the subset only contains on of the above two entries, the sum cannot be $0$ due to parity. If the subset contains both entries, since $1-2b_{k}+2b_{k-1}-1=2(b_{k-1}-b_{k})$, the sum cannot be $0$ again due to the assumption on $\theta_0$. Thus $\theta$ satisfies the assumption of Lemma \ref{criterion}.
    \end{proof}
    \begin{rem}
      This proposition implies both $M_\theta$ and $M_{\theta_0}$ are fine moduli space of stable representations.
    \end{rem}
    \begin{rem}
      For the rest of this paper, when $\theta$ and $\theta_0$ show up together, we always assume $\theta$ is obtained from $\theta_0$ from the above procedure.
    \end{rem}

    \section{Construction of morphism between moduli spaces}We start by constructing a morphism between representation schemes, i.e $F:\mathrm{Rep(Q)}\to \mathrm{Rep(Q_0)}$. Since both $\mathrm{Rep(Q)}$ and $\mathrm{Rep(Q_0)}$ are affine schemes, it suffice to construct a $\mathbf{k}$-algebra homomorphism between their coordinate rings. We notice for each $a\in Q_{0,ar}$ from $i$ to $j$, it corresponds to an element $s$ in $\Hom_{\cO_{S_0}}(L_i,L_j)$, it is natural to consider mapping it to the element in $\mathbf{k}[b\in Q_{ar}]/J$ which corresponds to  $\pi^*s\in\Hom_{\cO_S}(\pi^*L_i,\pi^*L_j)$. The nontrivial part lies in the fact that such an element is not always of the form $b$ for $b\in Q_{ar}$.
      \begin{thm}\label{map}
      There is a natural morphism
      \begin{equation*}
          F:\mathrm{Rep(Q)}\to \mathrm{Rep(Q_0)}
      \end{equation*}
      Moreover, $F$ is surjective.
      \end{thm}
      \begin{proof}
      We define a $\mathbf{k}$-algebra homomorphism
      $\phi:\mathbf{k}[a\in Q_{0,ar}]/J_0\to \mathbf{k}[b\in Q_{ar}]/J$ as follows:
      For $a\in Q_{0,a}$, if $t(a)<k$ or $s(a)>k$, by the computation at the beginning of Section 3, we see there is a unique arrow $b$ in $Q$ corresponding to $a$, define $\phi(a)=b$.\
    
      If $t(a)=k$, then there is a unique arrow $b$ in $Q$ with $t(b)=k'$ corresponding to $a$, define $\phi(a)=b$. The case when $s(a)=k$ is handled similarly.\
    
      If $s(a)<k$ and $t(a)>k$, by Lemma \ref{cross arrow}, all but one such arrows represent sections whose divisor of zero contains $E$ as a component, and corresponds to a unique arrow $b$ in $Q$, define $\phi(a)=eb$. For the unique arrow $w$ that represents a section $s$ such that $div(s)-E$ is not effective, there is not arrow in $Q$ corresponding to it. By the proof of Lemma 3.4, we see the natural choice is to set $\phi(u)=u_{s(w)}u_{t(w)}$.
    
      It remains to check $\phi$ is well-defined, i.e elements in $J_0$ gets mapped into $J$ by $\phi$. It suffices to show we can choose a collection of generators of $I_0$, such that when we consider the generators as elements in $\mathbf{k}[a\in Q_{0,a}]$,their image under $\phi$ lies in $J$. The natural choice of generators consists of $\sum_{i=1}^nk_ip_i$ where $k_i\in \mathbf{k}^*$ and $p_i$'s are pairwise different paths that share the same source $i$ and target $j$ such so that $\sum_{i=1}^nk_ip_i$ corresponds to $0$ in $H^0(S_0,B_{i+1}+\ldots+B_j)$. By our choice above, it is clear that if either $i\geq k$ or $j\leq k$, $\phi(\sum_{i=1}^nk_ip_i)=0$. If $i<k<j$, then the unique arrow cannot show up in $\sum_{i=1}^nk_ip_i$ with nonzero coefficient, as it represents the unique dimension in $H^0(S_0,B_{i+1}+\ldots+B_j)$ that does not lie in the span of the rest of the arrows. Thus $\phi(\sum_{i=1}^nk_ip_i)$ is a multiple of $e$. By our construction $\phi(\sum_{i=1}^nk_ip_i)$ corresponds to $0$ in $H^0(X,B_{i+1}+\ldots+B_j)$ (Caution: This does not imply $\phi(\sum_{i=1}^nk_ip_i)\in J!$). The fact that $\phi(\sum_{i=1}^nk_ip_i)$ is a multiple of $e$ implies it is in the image of the inclusion $H^0(X,B_{i+1}+\ldots+B_j-E)\hookrightarrow H^0(X,B_{i+1}+\ldots+B_j)$, so $\phi(\sum_{i=1}^nk_ip_i)/e$ corresponds to $0$ in $H^0(X,B_{i+1}+\ldots+B_j-E)$, hence is a generator of $J$. So $\phi(\sum_{i=1}^nk_ip_i)=(\phi(\sum_{i=1}^nk_ip_i)/e)e\in J$.\\
      We now show $F$ is surjective. Given $R_0\in\mathrm{Rep(Q_0)}$, we set $r_e=1$ and see from above there exist a unique choice for the values of other arrows of $R$ if $F(R)=R_0$ . The fact that these values comes from a representations follows from the fact that $R_0$ is.
    \end{proof}
    
      The next proposition shows $F$ respects the $\mathrm{PGL(\mathbf{1})}$- action.
    
      \begin{prop}\label{descent}
      Let $R_1$,$R_2$ be two representations of $Q$ with dimension vector $\mathbf{1}$. Suppose $R_1\sim R_2$, via the element $(g_1,\ldots,g_{k-1},g_k,g_{k'},g_{k+1},\ldots,g_n)$, then $F(R_1)\sim F(R_2)$.
      \end{prop}
      \begin{proof}
      From the construction of $F$, one directly check the element
      \begin{equation*}
          (g_1g_k, g_2g_k, \ldots, g_{k-1}g_k, g_kg_{k'}, g_{k+1}g_{k'}, \ldots, g_ng_{k'})
      \end{equation*}
      provides the equivalence.
      \end{proof}
    
    We now consider the interaction between $F$ and stability conditions. We first make an easy but important remark.
    \begin{rem}
      The $\mathrm{PGL(\mathbf{1})}$-action is compatible with stability, i.e. if $R_1\sim R_2$ and $R_1$ is $\theta$-semistable, then so is $R_2$.
    \end{rem}
    We let $U=\mathrm{Rep(Q)}-\mathbf{V}(e)$, this is the open subset of $\mathrm{Rep(Q)}$ consisting of representations where the value of the arrow $e$ is not $0$.
    \begin{prop}\label{on U}
      Suppose $F(R)=R'$, $R\in U$, R is $\theta$-stable, then $R'$ is $\theta_0$-stable.
    \end{prop}
    \begin{proof}
      Note if $$\theta_0=(-b_1,b_1-b_2,\ldots,b_{n-2}-b_{n-1}, b_{n-1})$$ as in Proposition \ref{weight} then
      \begin{equation*}
          \theta=\big(-2b_1,2(b_1-b_2),\ldots,2(b_{k-2}-b_{k-1}),1-2b_{k},2b_{k-1}-1,2(b_{k}-b_{k+1}),\ldots,2(b_{n-2}-b_{n-1}),2b_{n-1}\big)
    \end{equation*}
    We mention that $(1-2b_{k})+(2b_{k-1}-1)=2(b_{k-1}-b_{k})$. \
    
    Given $S'\subset R'$, $\vec{d'}$ be the dimension vector of $S'$. We claim that we can find a subrepresentation $S$ of $R$
    whose dimension vector $\vec{d}$ is as follows:
    \begin{enumerate}
        \item If $i<k$, $d_i=d'_i$. 
        \item If $j>k$, $d_j=d'_j$.
        \item $d_k=d_{k'}=d'_k$
    \end{enumerate}
    We now verify $\vec{d}$ indeed gives a subrepresentation of $R$ using Lemma \ref{subrep}. Let $a$ be an arrow from $i$ to $j$ such that $d_i=1$ and $d_j=0$, we want to show $r_a=0$.\ 
    
    If $j<k$, then $d'_i=1$, $d'_j=0$ and by construction of $F$, we see $r'_a=r_a=0$.\
    
    If $j=k$, then $d'_i=1$, $d'_k=0$. Then there is an arrow $a'$ from $i$ to $k$ in $Q_0$ such that $r'_{a'}=r_ar_e$ by the construction of $F$. Since $R\in U$, $r_e\neq 0$. But $r'_{a'}=0$, so we must have $r_a=0$.\ 
    
    If $j=k'$, then $d'_i=1$, $d'_k=0$, there is an arrow $a'$ in $Q_0$ such that $r_a=r'_{a'}=0$.\ 
    
    If $j>k$, and $i=k,k'$ or $i>k$, the above arguments applies. When $i<k$, then there is an arrow $a'$ in $Q_0$ such that $0=r'_{a'}=r_ar_e$ again as $r_e\neq 0$, we get $r_a=0$.\ 
    
    So a subrepresentation $S$ of $R$ with dimension vector $\vec{d}$ exists. It is clear by our choice of dimension vector 
    \begin{align*}
        \theta(S)=2\theta_0(S')
    \end{align*}
    Since $R$ is $\theta$-stable, $\theta(S)>0$, thus $\theta_0(S')>0$. Apply this argument for all $S'\subset R'$, we see $R'=F(R)$ is $\theta_0$-stable.
      \end{proof}
    
    The next proposition shows $F$ is injective.
    \begin{prop}\label{injective}
      Suppose $R_1,R_2\in U$, and $F(R_1)\sim F(R_2)$ under the action of $(g_1,\ldots,g_n)$, then $R_1\sim R_2$.
    \end{prop}
    \begin{proof}
      Let $e_i$ denote the value of $e$ in $R_i$ for $i=1,2$, then $e_1e_2\neq 0$. Again by the construction of $F$, one directly checks that 
      \begin{equation*}
          \Big(g_1e_2,g_2e_2,\ldots,g_{k-1}e_2,g_ke_1,g_ke_2,g_{k+1}e_1,\ldots,g_ne_1\Big)
      \end{equation*}
      provides the equivalence.
    \end{proof}
    
    Any full strong exceptional toric system coming from standard augmentation from $\mathbb{P}^2$ satisfying (\ref{technical}) is of the form $\{\ldots,H-\Delta,\ldots,H-\Delta',\ldots,H-\Delta''\}$ where $\Delta,\Delta',\Delta''$ are (possibly empty) sum of at most $3$ exceptional divisors, and the terms in $\ldots$ are single exceptional divisor $E_1$ or $E_1-E_2$. Note all divisors in the above toric system are required to be slo except the last one. We will use these restrictions when we exhaust all possible forms of toric systems for $\mathcal{TS}^{op}$ (see Remark \ref{op}).
    
     The next proposition provides structural properties of representations with $r_e=0$.
     
    \begin{prop}\label{structure}
      Suppose $R\in \mathbf{V}(e)$ and $R$ is $\theta$-stable, then for all $i<k$
      \begin{align*}
          r_{u_i}\neq 0
      \end{align*}
      Also for all $j>k$,
      \begin{align*}
          r_{u_j}\neq 0.
      \end{align*}
    \end{prop}
    \begin{rem}
      The strategy for the proof is to show that if a bad arrow violating the statement exists, one can construct a subrepresentation $S\subset R$ such that $\theta(S)< 0$, contradicting the fact that $R$ is stable. The proof itself uses case by case study and is a bit tedious. We direct the reader to the generalized version of this paper \cite{QZ} for it.
    \end{rem}
    
    We illustrate the idea of proof of Proposition \ref{structure} by an example.
    \begin{exmp}
        Let $S=Bl_p\mathbb{P}^2$ with $E$ the exceptional divisor. Consider the exceptional collection $\{\cO_S,\cO_S(H-E),\cO_S(H),\cO_S(2H-E)\}$. The quiver of sections looks like
        
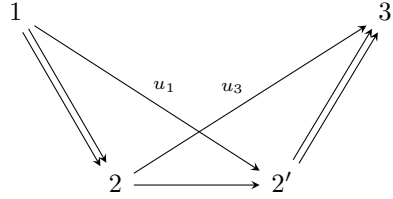
\[
\begin{tikzcd}[arrow style=tikz,>=stealth,row sep=5em]
1 
   \arrow[dr,shift right=.3ex]
   \arrow[dr,shift left=.3ex]
   \arrow{drrr}{u_1}
&&&&3
 \\         
&2 \arrow{urrr}{u_3}
    \arrow[rr]
&&2' \arrow[ur, shift left=.3ex]
    \arrow[ur,shift right=.3ex]
\end{tikzcd}
\]
By our choice of weight, $\theta=(-2a,1-2b,2a-1,2b)$ for $a,b>0$ (that is it has toric form $(2a,2a+2b-1,2b)$). Suppose $R\in \mathbf{V}(e)$ and if $r_{u_1}=0$, then we have a subrepresentation $R'$ with dimension vector $(1,1,0,1)$ and 
\begin{align*}
    \theta(S')=-2a+(1-2b)+2b=1-2a<0
\end{align*}
while if $r_{u_3}=0$, we have $R'$ with dimension vector $(0,1,0,0)$ and
\begin{align*}
    \theta(S')=1-2b<0
\end{align*}
So in both cases $S$ is unstable.
    \end{exmp}
      The following proposition generalizes Proposition \ref{on U}. 
      \begin{prop}\label{compatible}
      Suppose $F(R)=R'$, $R$ is $\theta$-stable, then $R'$ is $\theta_0$-stable.
      \end{prop}
      \begin{proof}
      If $R\in U$, the result follows from Proposition \ref{on U}. Suppose $R\in \mathbf{V}(e)$, we will use the same recipe of the proof of the Proposition \ref{on U}. Let $S'\subset R'$, $\vec{d'}$ be the dimension vector of $S'$, we claim there is a subrepresentation $S$ of $R$ whose dimension vector $\vec{d}$ is specifies as follows:
      \begin{enumerate}
        \item If $i<k$, $d_i=d'_i$. 
        \item If $j>k$, $d_j=d'_j$.
        \item $d_k=d_{k'}=d'_k$
    \end{enumerate}
    To prove the claim, we check the conditions in Lemma \ref{subrep}. Let $a$ be an arrow from $i$ to $j$ such that $d_i=1$ and $d_j=0$. If $j<k$, then $d'_i=1$, $d'_j=0$ and by construction of $F$, if $a'$ is the corresponding arrow in $Q_0$ to $a$, we see $r'_{a'}=r_a=0$.\
    
    If $j=k$ or $j=k'$, then $d'_k=0$, $d'_i=1$. By Proposition \ref{structure}, $r_{u_i}\neq 0$. Thus in $R'$, $r_{w_i}\neq 0$ ($w_i$ is the corresponding arrow in $Q_0$ defined in Definition \ref{arrow1}), this leads to a contradiction. So if $d_k=0$ or $d_{k'}=0$, we must have $d_i=0$ for all $i<k$.\ 
    
    If $j>k$, then if in addition $i>k$, then $d'_i=1$ and $d'_j=0$. So there is an arrow $a'$ in $Q_0$ such that $r_a=r'_{a'}=0$ by construction of $F$. If $i<k$, then $d'_i=1$ and $d'_j=0$. By Lemma \ref{cross arrow}, there is an arrow in $Q_0$ from $i$ to $j$ whose value is $r_{u_i}r_{u_j}\neq 0$, contradiction. Similarly, if $i=k'$ or $i=k$, then $d'_k=1$ and $d'_j=0$, we reach a contradiction following the argument in the previous paragraph.\ 
    
    So a subrepresentation $S$ of $R$ with dimension vector exists. Now it is clear that 
    \begin{align*}
        \theta(S)=2\theta_0(S')
    \end{align*}
    Since $R$ is $\theta$-stable, $\theta(S)>0$, thus $\theta_0(S')>0$. Apply this argument for all $S'\subset R'$, we see $R'=F(R)$ is $\theta_0$-stable.
    \end{proof}

    Let $C$ denote the closed subscheme of  $M_\theta$ containing  orbits of stable representations with $r_e=0$, i.e. 
    \begin{align*}
        C=\mathbf{V}(e)^{S}//\mathrm{PGL(\mathbf{1})}
    \end{align*}
    where $\mathbf{V}(e)^{S}$ is the open subscheme of $\mathbf{V}(e)$ consisting of stable representations.
    \begin{thm}\label{to a point}
      $$F(\mathbf{V}(e)^S)\in T_0(P)$$
      In other words, the image of a stable representation in $\mathbf{V}(e)$ under $F$ lies in the isomorphism class $T_0(P)\in M_{\theta_0}$.
    \end{thm}
    \begin{proof}
      Let $R$ be a representation of $Q$ such that $r_e=0$, then by results in Section 3, the construction of $F$ and Proposition \ref{structure}, $F(R)$ satisfies the following properties:
      \begin{enumerate}
          \item For any $i<k$, all the arrows in $H^0(S_0,B_i+\ldots+B_{k-1})$ passing through $P$ have value $0$.
          \item For each $i<k$, the unique arrow $w_i$ in $H^0(S_0,B_i+\ldots+B_{k-1})$ not passing through $P$ have nonzero value.
          \item For any $j>k$, all the arrows in $H^0(S_0,B_{k}+\ldots+B_{j-1})$ passing through $P$ have value $0$.
          \item For each $j>k$, the unique arrow $w_j$ in $H^0(S_0,B_{k}+\ldots+B_{j-1})$ not passing through $P$ have nonzero value.
    \end{enumerate}
    Claim: All representations of $Q_0$ satisfying the above properties are in the same $\mathrm{PGL(\mathbf{1})}$ orbit.\
    
    Let $R'$ be such a representation, then by letting $g_k=1$ and choose suitable $g_l$ for $l\neq k$ and replace $R'$ by $g\cdot R$ we can assume $r'_{w_i}=1$ for all $i$. We now show the value of all other arrows only depends on $Q_0$, instead of $R'$.\ 
    
    Let $a$ be an arrow in $Q$. If $t(a)<k$, then $w_{t(a)}\circ a$ is an arrow from $s(a)$ to $k$. Using the property above, we can see that:
      \begin{align*}
          r'_a=
          \begin{cases}
          0 &\text{if $div(a)$ passes through $P$}\\
          1 &\text{if $div(a)$ does not pass through $P$}
          \end{cases}
      \end{align*}
      Similarly, we can get the value for arrows with $s(a)<k$.\
    
    If $s(a)<k$ and $t(a)>k$. By Lemma \ref{cross arrow}, if the passes through $P$, then its value is $0$, otherwise $1$. This shows all such $R'$ satisfying the above properties are isomorphic to a representation $\mathfrak{R}$ whose values of arrows are given by for all $i$
    \begin{align*}
        \mathfrak{r}_{w_i}=1
    \end{align*}
    and for other arrows
    \begin{align*}
        \mathfrak{r}_{a}=
        \begin{cases}
        0 &\text{if $div(a)$ passes through $P$}\\
        1 &\text{if $div(a)$ does not pass through $P$}
        \end{cases}
    \end{align*}
    
    By Proposition \ref{compatible}, the uniqe orbit containing $F(\mathbf{V}(e)^S)$ is $\theta_0$-stable, so it corresponds to a point on $S_0$. On the other hand, it is clear that the any representative in the isomorphism class $T_0(P)=\bigoplus_{i=0}^n(L_i^\vee)_P$ satisfies the all the above properties, so the image must be in $T_0(P)$.
    \end{proof}

    \begin{prop}\label{to a point2}
      If $R$ is a representation in $\mathrm{Rep(Q)}$ such that $F(R)\in T_0(P)$ and $r_e\neq 0$,then $R$ is not $\theta$-stable.
    \end{prop}
    \begin{proof}
      Note $F(R)$ satisfies the 4 conditions in the proof of the previous theorem. Using the construction of $F$ and the fact that $r_e\neq 0$, we see that $R$ satisfies:
      \begin{enumerate}
          \item For any $i<k$, all arrows in from $i$ to $k$ have value $0$.
          \item For any $j>k$, all arrows from $k'$ to $j$ have value $0$.
          \item For any $i<k<j$, all arrows from $i$ to $j$ have value $0$. 
      \end{enumerate}
      Note the last property uses the same argument in the last paragraph of proof of Theorem.\ 
      Hence we see that $R$ has a subrepresentation $S$ defined by:
      \begin{align*}
          &\dim(S_1)=\dim(S_2)=\ldots=\dim(S_{k-1})=\dim(S_{k'})=1\\
          &\dim(S_k)=\dim(S_{k+1})=\ldots=\dim(S_n)=0
      \end{align*}
      We compute:
      \begin{align*}
          \theta(S)&=-2b_0+2(b_0-b_1)+\ldots+2(b_{k-2}-b_{k-1})+(2b_{k-1}-1)\\
          &=-1
      \end{align*}
      This shows $R$ is not $\theta$-stable.
    \end{proof}
    
    \begin{cor}\label{commute}
      The natural morphism
      \begin{equation*}
           F:\mathrm{Rep(Q)}\to \mathrm{Rep(Q_0)}
      \end{equation*}
      descends to a projective morphism
      \begin{equation*}
          f:M_{\theta}\to M_{\theta_0}
      \end{equation*}
      which fits into a commutative diagram
      \[ \begin{tikzcd}
      S \arrow[r,dashed,"T"] \arrow{d}{\pi} & M_\theta \arrow{d}{f} \\%
      S_0 \arrow{r}{T_0}& M_{\theta_0}
      \end{tikzcd}
      \]
      In other words, the above diagram is commutative wherever $T$ is defined.
      \end{cor}
      \begin{proof}
      By Proposition \ref{compatible}, we can restrict the domain of $F$ to get 
      \begin{equation*}
          F_0:\mathrm{Rep(Q)}^{S}_\theta\to \mathrm{Rep(Q_0)}^{S}_{\theta_0} 
      \end{equation*}
      Composed with the projection map $\mathrm{Rep(Q)}^{S}_\theta\to M_{\theta_0}$ of the geometric quotient, we obtain a morphism
      \begin{equation*}
          F_1:\mathrm{Rep(Q)}^{S}_\theta\to \mathrm{Rep(Q_0)}^{S}_{\theta_0}//\mathrm{PGL(\mathbf{1})}=M_{\theta_0}
      \end{equation*}
      Proposition \ref{descent} shows $F_1$ descends to the quotient, and give morphism
      \begin{equation*}
          f:M_{\theta}\to M_{\theta_0}
      \end{equation*}
      Since both $M_{\theta}$,$ M_{\theta_0}$ are projective schemes, $f$ is a projective morphism.\
      
      Suppose the tautological rational map $T$ is defined at a point $s\in S$ and $s$ does not lie on the exceptional curve $E$, then we know 
      \begin{align*}
          T(s)=&\bigoplus_{i=1}^{n-k}(\pi^*\cO(-D_{i}))_s\oplus(\pi^*\cO(-D_{n-k+1}+E))_s\oplus\\
          &(\pi^*\cO(-D_{n-k+1}))_s\oplus\bigoplus_{i=n-k+2}^n(\pi^*\cO(-D_{i}+E))_s
      \end{align*}
      is $\theta$-stable. Since $s\notin E$, there exist an neighbourhood $V$ of $s$ such that $\pi|_V=id$. We note that by projection formula
      \begin{align}
          A_0&=\mathrm{End}\big(\bigoplus_{i=1}^n\cO(-D_i)\big)\\
          &=\mathrm{End}\big(\bigoplus_{i=1}^n\pi^*(\cO(-D_i)\big)
      \end{align}
      Recall the construction of $F$, in the definition of ring homomorphism, for an arrow $a$ in $Q_0$ from $i$ to $j$ representing an section $s$ in 
      $\Hom_{\mathcal{O}_{S_0}}(L_i,L_j)$, we pulled it back by $\pi^*$ to get an element in $\Hom_{\mathcal{O}_S}(\pi^*L_p,\pi^*L_q)$, then find the corresponding arrow or multiple of arrows using the Lemmas in section 3. Hence tracing the definition of $F$ and the descent we have
      \begin{equation*}
          f(T(s))=\bigoplus_{i=1}^n(\pi^*\cO(-D_i))_s
      \end{equation*}
      whose $A_0$-module sturcture is given by (6.13) and (6.14). Since $\pi=id$ near $s$, we see $\bigoplus_{i=1}^n(\pi^*\cO(-D_i))_s$ as an $A_0$-module is isomorphic to $\bigoplus_{i=1}^n(\cO(-D_i))_{\pi(s)}=T_0(\pi(s))$.\
      
      Suppose $T$ is defined at a point $s\in E$, then $T(s)$ is $\theta$-stable and $T(s)\in \mathbf{V}(e)$. By Theorem \ref{to a point}, we see $f(T(s))=T_0(P)=T_0(\pi(s))$.
    \end{proof}
    
    \begin{prop}\label{nonempty}
      The tautological map $T$ is defined on a nonempty open subset of $S\backslash E$. In particular, $$M_\theta\neq \emptyset$$
    \end{prop}
    \begin{proof}
        Consider $R$ in the class $\bigoplus_{i=1}^{n-k}(\pi^*\cO(-D_{i}))_s\oplus(\pi^*\cO(-D_{n-k+1}+E))_s\oplus(\pi^*\cO(-D_{n-k+1}))_s\oplus\bigoplus_{i=n-k+2}^n(\pi^*\cO(-D_{i}+E))_s$ of representations of $(Q,I)$. For any arrow $a$ in $Q$, $r_a=0$ if $s\in div(a)$. Since $div(a)$ is a divisor for all arrows and there are finitely many arrows $a$ in $Q$, there exist an open set $\mathcal{U}\subset S$ such that for any $s\in \mathcal{U}$, $r_a\neq 0$ for all arrows $a\in Q_{ar}$. We take $s\in \mathcal{U}$. By Remark \ref{comparison}, this representation $\bigoplus_{i=1}^{n-k}(\pi^*\cO(-D_{i}))_s\oplus(\pi^*\cO(-D_{n-k+1}+E))_s\oplus(\pi^*\cO(-D_{n-k+1}))_s\oplus\bigoplus_{i=n-k+2}^n(\pi^*\cO(-D_{i}+E))_s$ will be the best candidate to be stable.\ 
        
        Let $s_0=\pi(s)$, then $R'=F(R)\in\bigoplus_{i=1}^n(\cO(-D_i))_{s_0}$. By definition of $F$, $r'_b\neq 0$ for all arrows $b\in Q_{0,ar}$. Since $M_{\theta_0}\neq \emptyset$, by Remark \ref{comparison},  $R'$ is $\theta_0$-semistable. Now let $S$ be a subrepresentation of $R$. We first assume that  $\dim(S_k)=\dim(S_{k'})$. We claim we can find a subrepresentation $S'\subset R'$ such that $\dim(S'_i)=\dim(S_i)$ for all $i\in\{1,2,\ldots,n\}$. To check $S'$ is well defined, by Lemma \ref{subrep}, it suffices to show for any $i<j$ with $\dim(S'_i)=1$ and $\dim(S'_j)=0$, there are no arrow from $i$ to $j$. By definition of $S'$, $\dim(S_i)=1$ and $\dim(S_j)=0$. By our choice of $s$, this is only possible when there are no arrows from $i$ to $j$. Since $E^2=-1$, either $j<k$ or $i>k$, in both cases the arrows from $i$ to $j$ in $Q$ are in bijection with arrows from $i$ to $j$ in $Q_0$, hence there is no arrows from $i$ to $j$ in $Q_0$. Hence $S'$ with the prescribed dimension is well-defined and we obtain
        \begin{align*}
            \theta(S)=2\theta_0(S')
        \end{align*}
        Since $R'$ is stable and $S'\subset R'$, we have $\theta_0(S')>0$, so $\theta(S)>0$.\ 
        
        Suppose $\dim(S_k)\neq\dim(S_{k'})$, then either $\dim(S_k)=1$ and $\dim(S_k')=0$, but this is not possible since such dimensions require $r_e=0$, or $\dim(S_k)=0$ and $\dim(S_k')=1$, in this case we claim either $\dim(S_i)=0$ for all $i<k$ or $\dim(S_j)=1$ for all $j>k$.\ 
        
        If the claim is not true, the there exist $i<k<j$ such that $\dim(S_i)=1$ and $\dim(S_j)=0$. By Proposition \ref{hom}, there is an arrow $a$ from $i$ to $j$, hence $r_a=0$, contradicting our choice of $s$.\ 
        
        If $\dim(S_i)=0$ for all $i<k$, then we see that $S^\#$ defined by $\dim(S^{\#}_l)=\dim(S_l)$ for all $i$ except $\dim(S^\#_{k'})=0$ is also a subrepresentation of $R$. Note $\theta(S)=\theta(S^\#)+(2a_k-1)>\theta(S^\#)$. Since $\dim(S^\#_k)=\dim(S^\#_{k'})=0$, we can apply the first part of this proof to show $\theta(S^\#)>0$. Hence $\theta(S)>0$.
        
        If $\dim(S_j)=1$ for all $j>k$, then we see that $S^\#$ defined by $\dim(S^\#_l)=\dim(S_l)$ for all $l$ except $\dim(S^\#_{k})=1$ is also a subrepresentation of $R$. Note $\theta(S)=\theta(S^\#)-(1-2b_{k-1})>\theta(S^\#)$. Since $\dim(S^\#_k)=\dim(S^\#_{k'})=1$, we can apply the first part of this proof to show $\theta(S^\#)>0$. Hence $\theta(S)>0$.\
        
        Thus we have shown that the tautological rational map $T$ is defined on $\mathcal{U}$, which is a nonempty open subset of $S\backslash E$. Moreover, $M_\theta$ contains the image of $T|_{\mathcal{U}}$, hence has to be nonempty.
    \end{proof}  
      
    \begin{cor}\label{isom on open}
      $f$ is a surjective morphism, $T$ is defined on $S\backslash E$. Moreover, $f$ induces an isomorphism between $M_\theta\backslash C$ and $S_0-\{P\}$ and $T$ induces an isomorphism between $S\backslash E$ and $M_\theta\backslash C$.
    \end{cor}
    
    \begin{proof}
      By Proposition \ref{nonempty}, $T$ is defined on $\mathcal{U}\subset S\backslash E$. By Corollary \ref{commute}, $T_0(\pi(\mathcal{U}))$ is in the image of $f$. Note $T_0(\pi(\mathcal{U}))$ is an open dense subset of $M_{\theta_0}$. Since $f$ is proper, the image is all of $M_{\theta_0}$.\ 
    
      We claim $T$ is defined on $S\backslash E$. Let $s\in S\backslash E$, and let $s_0=\pi(s)$. We identify $S_0$ and $M_{\theta_0}$ using $T_0$, then 
      \begin{equation*}
          T_0(s_0)=\bigoplus_{i=1}^n(\cO(-D_i))_{s_0}
      \end{equation*}
       lies in the image of $f$ since $f$ is surjective. Pick a representation $R$ of $Q$ so that $f([R])=T_0(s_0)$. Now consider the $A$-module
      \begin{equation*}
          \bigoplus_{i=1}^{n-k}(\pi^*\cO(-D_{i}))_s\oplus(\pi^*\cO(-D_{n-k+1}+E))_s\oplus(\pi^*\cO(-D_{n-k+1}))_s\oplus\bigoplus_{i=n-k+2}^n(\pi^*\cO(-D_{i}+E))_s
      \end{equation*}
    which we need to prove to be $\theta$-stable. Take a basis for each of the 1-dimensional direct summands and apply $F$, we see from the proof of Corollary \ref{commute} that 
    \begin{equation*}
        F\Big(\bigoplus_{i=1}^{n-k}(\pi^*\cO(-D_{i}))_s\oplus(\pi^*\cO(-D_{n-k+1}+E))_s\oplus(\pi^*\cO(-D_{n-k+1}))_s\oplus\bigoplus_{i=n-k+2}^n(\pi^*\cO(-D_{i}+E))_s\Big)
    \end{equation*}
    is isomorphic to $\bigoplus_{i=1}^n(\cO(-D_i))_{s_0}$.
    By Lemma \ref{injective}, this shows 
    \begin{equation*}
        \bigoplus_{i=1}^{n-k}(\pi^*\cO(-D_{i}))_s\oplus(\pi^*\cO(-D_{n-k+1}+E))_s\oplus(\pi^*\cO(-D_{n-k+1}))_s\oplus\bigoplus_{i=n-k+2}^n(\pi^*\cO(-D_{i}+E))_s\sim R
    \end{equation*}
    Since $R$ is $\theta$-stable, so is $$\bigoplus_{i=1}^{n-k}(\pi^*\cO(-D_{i}))_s\oplus(\pi^*\cO(-D_{n-k+1}+E))_s\oplus(\pi^*\cO(-D_{n-k+1}))_s\oplus\bigoplus_{i=n-k+2}^n(\pi^*\cO(-D_{i}+E))_s$$. Thus $T$ is defined at $s$. This applies to all $s\in S\backslash E$, so the domain of $T$ contains $S\backslash E$.
    
    Restricting to $S\backslash E$, we have commutative diagram
    \[ \begin{tikzcd}
      S\backslash E \arrow[r,"T|_{S\backslash E}"] \arrow{d}{\pi|_{S\backslash E}} & M_\theta\backslash C \arrow{d}{f|_{M_\theta\backslash C}} \\%
      S_0\backslash\{P\} \arrow{r}{T_0|_{S_0\backslash\{P\}}}& M_{\theta_0}\backslash\{P\}
      \end{tikzcd}
      \]
    Since $f(\mathbf{V}(e)=\{T_0(P)\}$, and $f$ is surjective, $f|_{M_\theta\backslash C}$ is also surjective. By Lemma \ref{injective}, $f|_{M_\theta\backslash C}$ is also injective, so $f|_{M_\theta\backslash C}$ is a bijection. This in turn imply $T|_{S\backslash E}$ is surjective, since $T_0|_{S_0\backslash\{P\}}\circ \pi|_{S\backslash E}$ is an isomorphism. Since $S\backslash E$ is irreducible, so is 
    $M_\theta\backslash C=T|_{S\backslash E}(S\backslash E)$. Recall $M_\theta\backslash C$ is also reduced and $f|_{M_\theta\backslash C}$ is a bijection mapping it to a normal variety, so we can apply Zariski Main Theorem \cite[Corollary 4.4.3]{Liu} to conclude that $f|_{M_\theta\backslash C}$ is an isomorphism. Then $T|_{S\backslash E}$ is an isomorphism from the commutative diagram above.
    \end{proof}
    By now we have constructed a surjective morphism $f:M_\theta\to M_{\theta_0}$ which shares many feature as the blow up math $\pi:S\to S_0$, in the next section we show they are in fact the same.
    
    \section{Analysis of the fibre}
        
    We summarize what we know about $C$ so far: 
    \begin{itemize}
        \item $C$ is nonempty since $M_\theta$ is projective.
        \item $C$ is proper since $f$ is proper and $C=f^{-1}(P)$ by Theorem \ref{to a point} and Proposition \ref{to a point2}.
    \end{itemize}

    \begin{thm}\label{except locus}
      In the setting as the beginning of this section, we have $C\cong \mathbb{P}^1$
    \end{thm}
    The idea of this proof is to reduce the $\mathrm{PGL(\mathbf{1})}$ action on $\mathbf{V}(e)$ to the scaling action of $\mathbf{k}^*$ on $\mathbb{A}^2-0$, thus showing the quotient is $\mathbb{P}^1$.\
    
    Let $R\in\mathbf{V}(e)$, by Proposition \ref{structure}, $r_{u_i}\neq 0$ for all $i$. Let 
    \begin{align*}
        g_\lambda=\Big(r_{u_1},r_{u_2},\ldots,r_{u_{k-1}},\lambda,1,\frac{\lambda}{r_{u_{k+1}}},\ldots,\frac{\lambda}{r_{u_n}})\Big)\in \mathrm{PGL(\mathbf{1})}
    \end{align*}
    where $\lambda\in\mathbf{k}^*$ is arbitrary,
    then $g_\lambda\cdot R\in \mathbf{V}(e)$ satisfies $r_{u_i}=1$ for all $i$. Moreover if $a\in Q_{ar}$, and $t(a)<k$, then $u_{t(a)}\circ a$ is an arrow from $s(a)$ to $k'$, thus if $u_{t(a)}\circ a $ is a multiple of $e$, which is equivalent to $div(a)$ passes through $P$, then $r_a=0$, otherwise $r_a\neq 0$. In the second case, we can replace $a$ by a scalar multiple of it and assume $r_a=1$. Thus
    \begin{align*}
        r_a=
        \begin{cases}
          0&\text{if $a$ passes through $P$}\\
          1&\text{if $a$ does not pass through $P$}
        \end{cases}
    \end{align*}
    Similarly, all arrows from with $s(a)>k$ have their values determined. We call such a representation in the normal form. A normalized representation $R\in \mathbf{V}(e)$ is determined by the values of the following three types of arrows
    \renewcommand{\labelenumi}{\roman{enumi}}
    \begin{enumerate}
        \item Indecomposible arrows $a$ with $t(a)=k$.
        \item Indecomposible arrows $a$ with $s(a)=k'$
        \item Indecomposible arrows $a$ with $s(a)<k$ and $t(a)>k'$.
    \end{enumerate}
    We note the subset of normalized representation in $\mathbf{V}(e)$ is an algebraic subset, more specifically, it is $\mathbf{V}(e,u_1-1,\ldots,u_n-1)$. For $\lambda\in\mathbf{k}^*$, let $\lambda$ acts on $\mathbf{V}(e,u_1-1,\ldots,u_n-1)$ by $g_\lambda$ as above, then it is clear from the definition of geometric quotient
    \begin{equation}
        \mathbf{V}(e,u_1-1,\ldots,u_n-1)^{S}//\mathbf{k}^*=\mathbf{V}(e)^{S}//\mathrm{PGL(\mathbf{1})}
    \end{equation}
    \begin{rem}
      From now on, we denote $\mathbf{V}(e,u_1-1,\ldots,u_n-1)$ by $W$ for simplicity of presentation.
    \end{rem}
 
    \begin{lem}\label{key}    
      A representation $R\in \mathbf{V}(e)$ satisfying $u_i=1$ for all $i$ is determined by the value of two arrows, in other words
      \begin{align*}
          \mathbf{V}(e,u_1-1,\ldots,u_n-1)\cong\mathbb{A}^2
     \end{align*}
    \end{lem}
    \begin{proof}
    Recall the points we blow up on $\mathbb{P}^2$ are in general position, hence we can prescribe coordinates to them without loss of generality. For the sake of presentation, we will provide full detail in the first few cases, then provide a list of three types of arrows and their relations in the other cases. The reader should be able to fill in the details with ease.\\

    \noindent\textbf{Part 1.} Suppose the two line divisors ($H-\cdots$) are on same side of $E$ in $\mathcal{TS}^{op}$. Without loss of generality, we can assume they are both on the left of $E$. \\
    Case 1: The second line divisor is of the form $H-E$. 
    If $\mathcal{TS}^{op}=\{H,H-E,E,H-E\}$. In this case, the undertermined arrows are of type i. Now $h^0(S,H-E)=2$. $h^0(S,2H-E)=5$ and by Lemma \ref{surj23} all arrows from vertex $1$ to vertex $3$ are decomposible. So the value of the two arrows from $2$ to $3$ determines the representation. Clearly the possible values forms an algebraic set
    $\mathbb{A}^2$.\ 
    
    If $\mathcal{TS}^{op}=\{E_1,H-E_1,H-E,E,H-E-E_1\}$,then $k=4$. In this case, the undermined arrows are of type i. Again $h^0(S,H-E)=2$. Note $B_1+B_2=E_1+(H-E_1)=H$, hence using the same argument as above, we see all arrows from vertex $1$ to vertex $4$ are decomposible. Since $E_1-E$ is not effective, by Lemma \ref{surj22}, all arrows from $2$ to $4$ are decomposible, thus the representation is determined by the value of the two arrows $a,b$ from $3$ to $4$.
    
    If $\mathcal{TS}^{op}=\{E_2,E_1-E_2,H-E_1,H-E,E,H-E-E_1-E_2\}$, we can apply the arguments in the previous subcase.\ 
    
    If $\mathcal{TS}^{op}=\{E_1-E_2,E_2,H-E_1-E_2,H-E,E,H-E-E_1\}$. Now $h^0(S,B_4)=2$. $h^0(S,B_3)=h^0(S,H-E_1-E_2)=1$. Also, $h^0(S,B_3+B_4)=3$. We see the unique arrow from $3$ to $4$ composed with the two arrows from $4$ to $5$ gives two distinct arrows from $3$ to $5$, thus there is an indecomposible arrow $c$ from $3$ to $5$.\ 
   Since $P$ is not on $E_1$, then $e_1\circ c$ represents an element in $H^0(S,H-E_1+H-E)$. Note $r_{e_i}=1$, then by Lemma \ref{surj22}, $c\circ e_1$ can be decomposed as composition of arrows from $1$ to $4$ with arrows from $4$ to $5$,
    the value of $r_c$ is again a fixed function of the two values of sections in $H^0(S,B_4)$.\\
    For arrows from $1,2$ to $5$, we use Lemma \ref{surj22} to see all of them must be decomposible. So the values of the two arrows from $4$ to $5$ gives the desired $\mathbb{A}^2$.\ 
    
    If $\mathcal{TS}^{op}=\{E_1-E_2,E_2-E_3,E_3,H-E_1-E_2-E_3,H-E,E,H-E-E_1\}$. In this case we only have arrows of type i. There are two arrows of type i from $5$ to $6$, which we call $f_1$ and $f_2$ and two arrows of type i from $4$ to $6$, which we call $g_1$ and $g_2$. We need to show there are two independent values among these four arrows. We can  assume without loss of generality that $E_1,E_2,E_3,E$ is obtained from blowing up $[1:0:0],[0:1:0],[0:0:1],[1,1,1]$ respectively. We can let $f_1$ represent the line $x-z$, $f_2$ represent the line $y-z$. $g_1$ represent the quadric $x(y-z)$ and $g_2$ represent the quadric $y(x-z)$. We let $u_{04}$ represent the line $x$, $u_{14}$ represent $y$ and $u_{24}$ represent the line $z$, $u_4$ represent $x$ and $u_3$ represent $yx$. By looking at arrow from $1,2,3$ to $6$, we obtain three relations:
    \begin{align*}
        f_2\circ u_{15}&=g_1\circ e_1\\
        f_1\circ u_{25}&=g_2\circ e_2\\
        (f_1-f_2)\circ u_{35}&=(g_2-g_1)\circ e_3
    \end{align*}
    From the first two equations, we see the values of $g_1,g_2$ are determined by those of $f_1$ and $f_2$ since $r_{e_2}\neq 0$,$r_{e_3}\neq 0$. It remains to check the third equation does not give extra information on the values of $f_1,f_2,g_1,g_2$. Using the following relations
    \begin{align*}
        u_5\circ u_{25}&=u_4\circ e_2\\
        u_4\circ e_1&=u_5\circ u_{15}+e\circ(f_2-f_1)\circ e\\
        u_5\circ u_{35}&=u_4\circ e_3-e\circ g_1\circ e_3
    \end{align*}
    and the fact that $r_e=0$, we obtain
    \begin{align*}
        (r_{f_1}-r_{f_2})r_{u_{35}}&=r_gr_{e_2}r_{u_{35}}/r_{u_{25}}-r_{g_1}r_{e_1}r_{u_{35}}/r_{u_{15}}\\
        &=r_{g_2}r_{u_5}r_{u_{35}}/r_{u_4}-r_{g_1}r_{u_5}r_{u_{35}}/r_{u_4}\\
        &=(r_{g_2}-r_{g_1})r_{e_3}
    \end{align*}
    hence the third equation is always true regardless of the values of $f_1,f_2$ so the value of $f_1$ and $f_2$ determines the representation.\\
    For the rest of the cases in Part 1, we will present them in diagrams.

    Case 2: The second line divisor is of the form $H-E-E_1$.\\
    
    \begin{table}[h]
    \begin{center}
\begin{tabular}{|l|l|l|l|l|}
\hline
$\mathcal{TS}^{op}$           & Type i             & Type ii                  & Type iii   & Relation                \\
\hline

$\{H-E_1,E_1,H-E-E_1$ & $a:2\to 4$& $\emptyset$ &$\emptyset$ & \\
$E,H-E\}$ &$b:3\to4$&&&\\
\hline
$\{H-E_1-E_2,E_2,E_1-E_2,$ & $a:3\to 5$& $\emptyset$ &$\emptyset$ & \\
$H-E-E_1,E,H-E\}$ &$b:4\to5$&&&\\
\hline

$\{E_2,H-E_1-E_2,E_1,$ & $a:3\to 5$& $\emptyset$ &$\emptyset$ & \\
$H-E-E_1,E,H-E-E_2\}$ &$b:4\to5$&&&\\
\hline

$\{E_3,E_2-E_3,H-E_1-E_2,E_1,$ & $a:4\to 6$& $\emptyset$ &$\emptyset$ & \\
$H-E-E_1,E,H-E-E_2-E_3\}$ &$b:5\to6$&&&\\
\hline

$\{E_3,H-E_1-E_2-E_3,E_2$ & $a:4\to 6$& $\emptyset$ &$\emptyset$ & \\
$E_1-E_2,H-E-E_1,E$ &$b:5\to6$&&&\\
$H-E-E_3\}$&&&&\\
\hline

$\{E_2-E_3,E_3,H-E_1-E_2-E_3,$ & $a:3\to 6$& $\emptyset$ &$\emptyset$ &$c$ determined  \\
$E_1,H-E-E_1,E,H-E-E_2\}$ &$b:4\to6$&&&by $a,b$\\
&$c:5\to 6$&&&\\
\hline

$\{H-E_1-E_2-E_3,E_3,E_2-E_3$ & $a:3\to 6$& $\emptyset$ &$\emptyset$ &$c$ determined  \\
$E_1-E_2,H-E-E_1,E$ &$b:4\to6$&&&by $a,b$\\
$H-E\}$&$c:5\to 6$&&&\\
\hline

$\{H,H-E-E_1,E,$ & $a:1\to 3$& $\emptyset$ & $c:2\to 4$ & $c$
determined\\
$E_1-E,H-E_1\}$ &$b:2\to3$&&&by $a,b$\\
\hline
$\{E_2,H-E_2,H-E-E_1,$           &       $a:2\to4$             &                $\emptyset$          &  $c:3\to 5$      &                   $c$
determined      \\
    $E,E_1-E,H-E_1-E_2\}$                          &             $b:3\to 4$       &                          &            &by $a,b$\\
\hline
$\{E_3,E_2-E_3,H-E_2,$&$a:4\to 5$&$\emptyset$&$d:4\to 6$& All determined\\
$H-E-E_1,E,E_1-E,$ & $b:3\to 5$&&&by $a,b$\\
$H-E_1-E_2-E_3\}$&$c:2\to 5$ &&&\\
\hline
$\{H-E_1-E_2-E_3,E_3,E_2-E_3,$&$a:3\to 6$&$\emptyset$& $\emptyset$&\\
$E_1-E_2,H-E-E_1,E,$ & $b:4\to 6$&&&\\
$H-E\}$&$c:5\to 6$ &&&\\
\hline
\end{tabular}
\end{center}
\end{table}
\pagebreak

 Case 3: The second line divisor is of the form $H-E-E_1-E_2$.\\
        
    \begin{table}[h]
    \begin{center}
\begin{tabular}{|l|l|l|l|l|}
\hline
$\mathcal{TS}^{op}$           & Type i             & Type ii                  & Type iii   & Relation                \\
\hline
$\{H-E_1,E_1-E_2,E_2$ & $a:2\to 5$& $\emptyset$ & $\emptyset$ &\\
$H-E-E_1-E_2,E,H-E\}$ &$b:3\to5$&&&\\
\hline
$\{E_3,H-E_1-E_3,E_1-E_2,E_2,$           &       $a:4\to6$             &                $\emptyset$          &  $\emptyset$      &                 \\
    $H-E-E_1-E_2,E,H-E-E_3\}$                          &             $b:3\to 6$       &                          &            &\\
\hline

$\{E_4,E_3-E_4,H-E_1-E_3$           &       $a:4\to7$             &                $\emptyset$          &  $\emptyset$      &                 \\
    $E_1-E_2,E_2,H-E-E_1-E_2$                          &             $b:5\to 7$       &                          &            &\\
    $E,H-E-E_3-E_4\}$&&&&\\
\hline

$\{E_3-E_4,E_4,H-E_1-E_3-E_4$           &       $a:4\to7$             &                $\emptyset$          &  $\emptyset$      &$c$ determined                 \\
    $E_1-E_2,E_2,H-E-E_1-E_2$                          &             $b:5\to 7$       &                          &            &by $a,b$\\
    $E,H-E-E_3-E_4\}$&$c:3\to 7$&&&\\
\hline

$\{H,H-E-E_1-E_2,E,$&$a:1\to 3$&$\emptyset$&$d:2\to 4$& All determined \\
$E_1-E,E_2-E_1,H-E_2\}$ & $b:1\to 3$&&$e:2\to 5$&by $a,b$\\
&$c:1\to 3$&&&\\
\hline

$\{E_3,H-E_3,H-E-E_1-E_2,$&$a:2\to 4$&$\emptyset$&$d:3\to 5$& All determined \\
$E,E_1-E,E_2-E_1$ & $b:2\to 4$&&$e:3\to 6$&by $a,b$\\
$H-E_2-E_3\}$&$c:1\to 4$&&&\\
\hline

$\{E_4,E_3-E_4,H-E_3,$&$a:2\to 5$&$\emptyset$&$e:4\to 6$& All determined \\
$H-E-E_1-E_2,E,E_1-E,$ & $b:2\to 5$&&$f:4\to 7$&by $a,b$\\
$E_2-E_1,H-E_2-E_3-E_4\}$&$c,d:3\to 5$&&&\\
\hline

$\{E_3-E_4,E_4,H-E_3-E_4,$&$c:1\to 5$&$\emptyset$&$a:4\to 6$& All determined \\
$H-E-E_1-E_2,E,E_1-E,$ & $d:2\to 5$&&$b:4\to 7$&by $a,b$\\
$E_2-E_1,H-E_2\}$&$f:3\to 5$&&&\\
\hline

$\{E_3-E_4,E_4-E_5,E_5,$&$a:1\to 6$&$\emptyset$&$d:5\to 7$& All determined \\
$H-E_3-E_4-E_5,H-E-E_1-E_2,$ & $b:2\to6$&&$e:5\to 8$&by $a,b$\\
$E,E_1-E,E_2-E_1,H-E_2\}$&$c:3\to 6$&&$f:4\to 7$&\\
&&&$g:4\to 7$&\\
\hline

$H-E_2,E_2,H-E-E_1-E_2$&$a:1\to 4$&$\emptyset$& $g:3\to 5$&All determined\\
$E,E_1-E,H-E_1$&$b:2\to 4$&&&by $a,b$\\
\hline 
$E_3,H-E_2-E_3,E_2,$&$a:2\to5$&$\emptyset$ &$c:4\to6$&All determined\\
$H-E-E_1-E_2,E,$&$b:3\to5$&&&by $a,b$\\
$E_1-E,H-E_1-E_3$&&&&\\
\hline
$\{E_4, E_3-E_4,H-E_2-E_3,$&$a:2\to6$& $\emptyset$&$c:5\to 7$&All determined\\
$E_2,H-E-E_1-E_2,E,$& $b:3\to 6$&&&by $a,b$\\
$E_1-E,H-E_1-E_3-E_4\}$ & $c:4\to 6$ &&&\\
\hline

$\{H-E_2-E_3,E_3,E_2-E_3,$&$a:1\to5$&&$c:3\to 6$&All determined\\
$H-E-E_1-E_2,E,E_1-E,$& $b:2\to 5$&&$d:4\to 6$&by $a,b$\\
$H-E_1\}$ &&&&\\
\hline

$\{E_4,H-E_2-E_3-E_4,E_3,$&$a:2\to6$&&$c:4\to 7$&All determined\\
$E_2-E_3,H-E-E_1-E_2,E,$& $b:3\to 6$&&$d:5\to 7$&by $a,b$\\
$E_1-E,H-E_1\}$ &&&&\\
\hline

$\{E_3-E_4,E_4,H-E_2-E_3-E_4,$&$a:3\to6$&&$c:5\to 7$&All determined\\
$E_2,H-E-E_1-E_2,E,$& $b:4\to 6$&&&by $a,b$\\
$E_1-E,H-E_1-E_3\}$ &&&&\\
\hline
\end{tabular}
\end{center}
\end{table}

    \begin{table}[h]
    \begin{center}
\begin{tabular}{|l|l|l|l|l|}
\hline
$\mathcal{TS}^{op}$           & Type i             & Type ii                  & Type iii   & Relation                \\
\hline

$\{H-E_2-E_3-E_4,E_4,E_3-E_4,$&$a:1\to6$&&$c:3\to 7$&All determined\\
$E_2-E_3,H-E-E_1-E_2,E,$& $b:2\to 6$&&$d:4\to 7$&by $a,b$\\
$E_1-E,H-E_1\}$ &&&$e:5\to 7$&\\
\hline

\end{tabular}
\end{center}
\end{table}
    Case 4: The second line divisor is of the form $H-E_1$ where $E_1\neq E$.\\
    
            \begin{table}[h]
        \begin{center}
    \begin{tabular}{|l|l|l|l|l|}
    \hline
    $\mathcal{TS}^{op}$           & Type i             & Type ii                  & Type iii   & Relation                \\
    \hline
    $\{H,H-E_1,E_1-E$ & $a:2\to 4$& $\emptyset$ & $\emptyset$ &\\
    $E,H-E-E_1\}$ &$b:2\to4$&&&\\
    \hline
    $\{H,H-E_1,E_1-E_2$ & $a:2\to 5$& $\emptyset$ & $\emptyset$ &\\
    $E_2-E,E,H-E-E_1-E_2\}$ &$b:2\to5$&&&\\
    \hline
    $\{E_2,H-E_2,H-E_1,$           &       $a:3\to5$             &                $\emptyset$          &  $\emptyset$      &                 \\
        $E_1-E,E,H-E_1-E_2-E\}$                          &             $b:3\to 5$       &                          &            &\\
    \hline
    $\{E_2-E_3,E_3,H-E_2-E_3,$ & $a:4\to 6$& $\emptyset$ & $\emptyset$ &$c$ determined\\
    $H-E_1,E_1-E,E,$&$b:4\to 6$&&&by $a,b$\\
    $H-E_1-E_2-E\}$ &$c:3\to6$&&&\\
    \hline
    $\{E_2-E_3,E_3-E_4,E_4,$ & $a:4\to 7$& $\emptyset$ & $\emptyset$ &$a,b$ determined\\
    $H-E_2-E_3-E_4,H-E_1,E_1-E$&$b:4\to 7$&&&by $c,d$\\
    $E,H-E_1-E_2-E\}$ &$c:5\to7$&&&\\
    &$d:5\to 7$&&&\\
    \hline

    \end{tabular}
    \end{center}
    \end{table}
    Case 5: The second line divisor is of the form $H-E_1-E_2$ where $E_1$,$E_2$ are different from $E$.\ In all the subcases, there will be no arrows of type ii or iii, so we only need to analyze arrows of type i.\ 
    
            \begin{table}[h]
    \begin{center}
\begin{tabular}{|l|l|l|l|l|}
\hline
$\mathcal{TS}^{op}$           & Type i             & Type ii                  & Type iii   & Relation                \\
\hline
$\{H-E_2,E_2,H-E_1-E_2,$ & $a:2\to 5$& $\emptyset$ & $\emptyset$ &\\
$E_1-E,E,H-E_1-E\}$ &$b:3\to5$&&&\\
\hline
$\{H-E_2,E_2,H-E_1-E_2,$           &       $a:2\to6$             &                $\emptyset$          &  $\emptyset$      &                 \\
    $E_1-E_4,E_4-E,E$                          &             $b:3\to 6$       &                          &            &\\
$H-E_1-E-E_4\}$&&&&\\
\hline
$\{E_3,H-E_2-E_3,E_2$           &       $a:3\to6$             &                $\emptyset$          &  $\emptyset$      &                 \\
    $H-E_1-E_2,E_1-E,E$                          &             $b:4\to 6$       &                          &            &\\
$H-E_1-E-E_3\}$&&&&\\
\hline
$\{H-E_2-E_3,E_3,E_2-E_3,$           &       $a:3\to6$             &                $\emptyset$          &  $\emptyset$      &                 \\
    $H-E_1-E_2,E_1-E,E,$                          &             $b:4\to 6$       &                          &            &\\
$H-E-E_1\}$&&&&\\
\hline
$\{H-E_2-E_3,E_3,E_2-E_3,$           &       $a:3\to7$             &                $\emptyset$          &  $\emptyset$      &                 \\
    $H-E_1-E_2,E_1-E_4,E_4-E,$                          &             $b:4\to 7$       &                          &            &\\
$E,H-E-E_1-E_4\}$&&&&\\
\hline
$\{H-E_2-E_3-E_4,E_4,E_3-E_4$           &       $a:3\to7$             &                $\emptyset$          &  $\emptyset$      &$c$ determined              \\
    $E_2-E_3,H-E_1-E_2,E_1-E,$                          &             $b:4\to 7$       &                          &            &by $a,b$\\
$E,H-E-E_1\}$&$c:5\to 7$&&&\\
\hline
\end{tabular}
\end{center}
\end{table}

    \begin{table}[h]
    \begin{center}
\begin{tabular}{|l|l|l|l|l|}
\hline

$\{H-E_2-E_3-E_4,E_4,E_3-E_4$           &       $a:3\to8$             &                $\emptyset$          &  $\emptyset$      &$c$ determined              \\
    $E_2-E_3,H-E_1-E_2,E_1-E_5,$                          &             $b:4\to 8$       &                          &            &by $a,b$\\
$E_5-E,H-E-E_1-E_5\}$&$c:5\to 8$&&&\\
\hline
\end{tabular}
\end{center}
\end{table}
\pagebreak
    Case 6: The second line divisor is of the form $H-E_1-E_2-E_3$ where $E_1,E_2,E_3$ are different from $E$.\ 
        
            \begin{table}[h]
    \begin{center}
\begin{tabular}{|l|l|l|l|l|}
\hline
$\mathcal{TS}^{op}$           & Type i             & Type ii                  & Type iii   & Relation                \\
\hline
$\{H-E_2,E_2-E_3,E_3,$&$a:2\to 6$& $\emptyset$ & $\emptyset$ &\\
$H-E_1-E_2-E_3,E_1-E,E,$ &$b:3\to6$&&&\\
$H-E_1-E\}$&&&&\\
\hline
$\{H-E_2,E_2-E_3,E_3,$&$a:2\to 7$& $\emptyset$ & $\emptyset$ &\\
$H-E_1-E_2-E_3,E_1-E_4,E_4-E,$ &$b:3\to7$&&&\\
$E,H-E_1-E_4-E\}$&&&&\\
\hline
$\{E_4,H-E_2-E_4,E_2-E_3,$           &       $a:3\to7$             &                $\emptyset$          &  $\emptyset$      &                 \\
    $E_3,H-E_1-E_2-E_3,E_1-E,$                          &             $b:4\to 7$       &                          &            &\\
$H-E_1-E-E_4\}$&&&&\\
\hline
$\{E_4-E_5,E_5,H-E_2-E_4-E_5,$           &       $a:3\to8$             &                $\emptyset$          &  $\emptyset$      &    All determined        \\
    $E_2-E_3,E_3,H-E_1-E_2-E_3,$                          &             $b:4\to 8$       &                          &            &by $a,b$\\
$E_1-E,E,H-E_1-E-E_4\}$&$c:5\to8$&&&\\
\hline

\end{tabular}
\end{center}
\end{table}
    \noindent\textbf{Part 2.} Suppose the two line divisors are on different side of $E$. We will use symmetry to reduce the case discussed.\
    
    We first consider the cases when $E$ is adjacent to the two line divisors in the toric system. The easiest case is $\mathcal{TS}^{op}=\{H-E,E,H-E,H\}$. This comes from blowing up a point on $\mathbb{P}^2$. Without loss of generality, we assume $P=[0:0:1]$. Then the two sections of $H-E$ are $x,y$. We denote the two arrows from $1$ to $2$ by $x_1$,$y_1$, the unique indecomposible arrow from $1$ to $2'$ by $z_1$, the unique indecomposible arrow from $2$ to $3$ by $z_2$ and the two arrows from $2'$ to $3$ by $x_2$, $y_2$. Note $r_{z_1}=r_{z_2}=1$.Moreover, we have the relation:
    \begin{align*}
        x_1z_2&=x_2z_1\\
        y_1z_2&=y_2z_1
    \end{align*}
    Thus the values of $x_1$,$y_1$ determine the value of $x_2$,$y_2$. By Lemma \ref{surj23}, there are no arrows of type iii. So the value of $x_1,x_2$ provides the desired $\mathbb{A}^2$.\ 
    
    If $\mathcal{TS}^{op}=\{E_1,H-E-E_1,E,H-E,H-E_1\}$. There are two arrows of type i, one is from $1$ to $3$, which we call $f_2$, another is from $2$ to $3$, which we call $f_1$. Then $f_1$ represents the unique section of $H-E-E_1$.There are two arrows from $3'$ to $4$ of type ii. We denote them by $f_2$, $g_2$. Note $f_2$, $g_2$ also represents sections of $H-E$. Moreover, $u_1,u_4$ both represent sections in $H^0(S,H)$ which do not lie in the subspace $H^0(S,H-E)$, we let them represent the same section. So we can choose $f_2$, $g_2$ so that
    \begin{align*}
        f_2\circ u_1&=u_4\circ f_1\circ e_1\\
        g_2\circ u_1&=u_4\circ g_1
    \end{align*}
    Thus the value of $f_1$, $g_1$ determines the value of $f_2$,$g_2$. Moreover, it is easy to check there are no type iii arrows from $1$ to $4$. For arrows from $2$ to $4$, one can easily check $f_2\circ e \circ f_1, g_2\circ e\circ f_1,u_3\circ f_1,g_2\circ u_1$ spans all the arrows from $2$ to $4$, so there are no type iii arrows. 
        Hence the value of $f_1,g_1$ gives the desired $\mathbb{A}^2$.\ 
       \begin{table}[h]
    \begin{center}
\begin{tabular}{|l|l|l|l|l|}
\hline
$\mathcal{TS}^{op}$           & Type i             & Type ii                  & Type iii   & Relation                \\
\hline
$\{E_2,E_1-E_2,H-E-E_1,$           &       $a:2\to4$             &                $c:4'\to 5$          &  $\emptyset$      &         All determined        \\
$E,H-E,H-E_1-E_2\}$&$b:3\to 4$&$d:4'\to 5$&&by $a,b$\\
\hline
$\{E_2,E_1-E_2,H-E-E_1,$           &       $a:2\to4$             &                $c:4'\to 5$          &  $\emptyset$      &         All determined        \\
$E,H-E-E_3,E_3$&$b:3\to 4$&$d:4'\to 6$&&by $a,b$\\
$H-E_1-E_2-E_3\}$&&&&\\
\hline
$\{E_1-E_2,E_2,H-E-E_1-E_2,$&$a:1\to 4$& $c:4'\to 5$ & $e:3\to 5$ &All determined\\

$E,H-E,H-E_1\}$&$b:2\to 4$&$d:4'\to 5$&&by $a,b$\\
\hline

$\{E_1-E_2,E_2,H-E-E_1-E_2,$&$a:1\to 4$& $c:4'\to 5$ & $e:3\to 5$ &All determined\\

$E,H-E-E_3,E_3$&$b:2\to 4$&$d:4'\to 6$&&by $a,b$\\
$H-E_1-E_3\}$&&&&\\
\hline

$\{E_1-E_2,E_2,H-E-E_1-E_2,$&$a:1\to 4$& $c:4'\to 5$ & $e:3\to 5$ &All determined\\

$E,H-E-E_3,E_3-E_4$&$b:2\to 4$&$d:4'\to 6$&$f:3\to 6$&by $a,b$\\
$E_4, H-E_1-E_3-E_4\}$&&&&\\
\hline

$\{E_1-E_2,E_2,H-E-E_1-E_2,$&$a:1\to 4$& $c:4'\to 6$ & $e:3\to 5$ &All determined\\

$E,H-E-E_3-E_4,E_4$&$b:2\to 4$&$d:4'\to 7$&&by $a,b$\\
$E_3-E_4,H-E_1-E_3\}$&&&&\\
\hline

$\{E_3,E_2-E_3,E_1-E_2,$&$a:2\to 5$& $d:5'\to 6$ & $\emptyset$ &All determined \\
$H-E_1-E,E,H-E,$ &$b:3\to5$&$e:5'\to 6$&&by $a,b$\\
$H-E_1-E_2-E_3\}$&$c:4\to 5$ &&&\\
\hline
$\{E_1,H-E-E_1,E,,$           &       $a:1\to3$             &                $c:3'\to 4$          &  $\emptyset$      &         All determined        \\
$H-E-E_2,E_2,H-E_1-E_2\}$&$b:2\to 3$&$d:3'\to 5$&&by $a,b$\\
\hline

\end{tabular}
\end{center}
\end{table}
\\
    
    Now we consider the case when $E,E_1-E$ is a segment of the toric system. Note this also covers the possibility of $E_1-E,E$ by symmetry.\ 
    
           \begin{table}[h]
    \begin{center}
\begin{tabular}{|l|l|l|l|l|}
\hline
$\mathcal{TS}^{op}$           & Type i             & Type ii                  & Type iii   & Relation                \\
\hline
$\{H-E_1-E,E,E_1-E,,$           &       $c:1\to 2$             &                $a:2'\to 4$          &  $d:1\to 3$      &         All determined        \\
$H-E_1,H\}$&&$b:2'\to 4$&&by $a,b$\\
\hline
$\{H-E_1-E,E,E_1-E,$&$c:1\to 2$& $a:2'\to 4$ & $d:1\to 3$ &All determined\\

$H-E_1-E_2,E_2,H-E_2\}$&&$b:2'\to 5$&&by $a,b$\\
\hline

$\{E_2,H-E-E_1-E_2,E,$&$c:1\to 3$& $a:3'\to 5$ & $d:2\to 4$ &All determined\\

$E_1-E,H-E_1,H-E_2\}$&&$b:3'\to 5$&&by $a,b$\\
\hline

$\{E_2,H-E-E_1-E_2,E,$&$c:1\to 3$& $a:3'\to 5$ & $d:2\to 4$ &All determined\\

$E_1-E,H-E_1-E_3,E_3,$&&$b:3'\to 6$&&by $a,b$\\
$H-E_2\}$&&&&\\
\hline

$\{E_2,H-E-E_1-E_2,E,$&$c:1\to 3$& $a:3'\to 6$ & $d:2\to 4$ &All determined\\

$E_1-E,H-E_1-E_3-E_4,E_4,$&&$b:3'\to 7$&$e:2\to 5$&by $a,b$\\
$E_3-E_4,H-E_2-E_3\}$&&&&\\
\hline

$\{E_2,H-E-E_1-E_2,E,$&$c:1\to 3$& $a:3'\to 5$ & $d:2\to 4$ &All determined\\

$E_1-E,H-E_1-E_3,E_3-E_4,$&&$b:3'\to 6$&&by $a,b$\\
$E_4,H-E_2-E_3-E_4\}$&&&&\\
\hline
\end{tabular}
\end{center}
\end{table}

    \begin{table}[h]
    \begin{center}
\begin{tabular}{|l|l|l|l|l|}
\hline

$\{H-E-E_1,E,E_1-E,$&$c:1\to 2$& $a:2'\to 5$ & $d:1\to 3$ &All determined\\

$H-E_1-E_2-E_3,E_3,E_2-E_3,$&&$b:2'\to 6$&$e:1\to 4$&by $a,b$\\
$H-E_2\}$&&&&\\
\hline

$\{H-E-E_1,E,E_1-E,$&$c:1\to 2$& $a:2'\to 4$ & $d:1\to 3$ &All determined\\

$H-E_1-E_2,E_2-E_3,E_3,$&&$b:2'\to 5$&&by $a,b$\\
$H-E_2-E_3\}$&&&&\\

\hline

$\{H-E-E_1,E,E_1-E,$&$c:1\to 2$& $a:2'\to 4$ & $d:1\to 3$ &All determined\\

$H-E_1-E_2,E_2-E_3,E_3-E_4,$&&$b:2'\to 5$&&by $a,b$\\
$E_4,H-E_2-E_3\}$&&$e:2'\to6$&&\\
\hline

\end{tabular}
\end{center}
\end{table}
    
\pagebreak
    Lastly we consider the case when $E,E_1-E,E_2-E_1$ is a segment of the toric system. Note this also covers the case of $E_2-E_1,E_1-E,E$ by symmetry.\

           \begin{table}[h]
    \begin{center}
\begin{tabular}{|l|l|l|l|l|}
\hline
$\mathcal{TS}^{op}$           & Type i             & Type ii                  & Type iii   & Relation                \\
\hline
$\{H-E-E_1-E_2,E,E_1-E,$           &     $\emptyset$           &                $c:2'\to 5$          &  $a:1\to 3$      &         All determined        \\
$E_2-E_1,H-E_2,H\}$&&$d:2'\to 5$&$b:1\to 4$&by $a,b$\\
\hline
$\{H-E-E_1-E_2,E,E_1-E,,$&$\emptyset$& $a:2'\to 5$ & $c:1\to 3$ &All determined\\

$E_2-E_1,H-E_2-E_3,E_3,$&&$b:2'\to 6$&$d:1\to 4$&by $a,b$\\
$H-E_3\}$&&&&\\
\hline

$\{H-E-E_1-E_2,E,E_1-E,,$&$\emptyset$& $a:2'\to 5$ & $c:1\to 3$ &All determined\\

$E_2-E_1,H-E_2-E_3,E_3-E_4,$&&$b:2'\to 6$&$d:1\to 4$&by $a,b$\\
$E_4,H-E_3-E_4\}$&&&&\\
\hline

$\{H-E-E_1-E_2,E,E_1-E,,$&$\emptyset$& $a:2'\to 5$ & $c:1\to 3$ &All determined\\

$E_2-E_1,H-E_2-E_3,E_3-E_4,$&&$b:2'\to 6$&$d:1\to 4$&by $a,b$\\
$E_4-E_5,E_5,H-E_3-E_4-E_5\}$&&&&\\
\hline

$\{H-E-E_1-E_2,E,E_1-E,,$&$\emptyset$& $a:2'\to 6$ & $c:1\to 3$ &All determined\\

$E_2-E_1,H-E_2-E_3-E_4,E_4,$&&$b:2'\to 7$&$d:1\to 4$&by $a,b$\\
$E_3-E_4,H-E_3\}$&&&$e:1\to 5$&\\
\hline

\end{tabular}
\end{center}
\end{table}
By now we have discussed all possible cases, and this concludes the proof.
    \end{proof}

    \begin{proof}[Proof of Theorem \ref{except locus}]
      By our reduction before, we see 
      \begin{equation*}
        C=\mathbf{V}(e)^{S}//\mathrm{PGL(\mathbf{1})}=\mathbf{V}(e,u_1-1,\ldots,u_n-1)^{S}//\mathbf{k}^*
      \end{equation*}
      We note $C\neq \emptyset$ implies $\mathbf{V}(e,u_1-1,\ldots,u_n-1)^{S}\neq \emptyset$.
      Let $a$,$b$ be the two arrows in the above lemma. We claim if $R\in\mathbf{V}(e,u_1-1,\ldots,u_n-1)^{S}$satisfies $r_a=r_b=0$, $R$ is not $\theta$-stable. For such an $R$, we note it satisfies
      \begin{enumerate}
          \item For any $i<k$, all arrows in from $i$ to $k$ have value $0$.
          \item For any $j>k$, all arrows from $k'$ to $j$ have value $0$.
          \item For any $i<k<j$, all arrows from $i$ to $j$ have value $0$. 
      \end{enumerate}
      Hence we can use the subrepresentation of $R$ as in Prop \ref{to a point2} and show $R$ is not stable.\\
      For any other arrow $g$ of type i,ii or iii , we realize from the proof of the above lemma, that
      $$r_g=m_ar_a+m_br_b$$
      where $m_a,m_b\in\C$ and depends only on $(Q,I)$. Since the stability of such a representation depends only on the vanishing of the value of these finite number of arrows (Lemma \ref{subrep}), we see $\mathbf{V}(e,u_1-1,\ldots,u_n-1)^{S}$ is of the complement of a finite collection of lines through origin in $\mathbb{A}^2$.
      If the collection of lines is nonempty, then $C=\mathbf{V}(e,u_1-1,\ldots,u_n-1)^{S}//\mathbf{k}^*$ is the complement of finite nonzero number of points in $\mathbb{P}^1$, which contradicts the fact that $C$ is proper. Hence the collection of lines is empty and $C\cong\mathbb{P}^1$.
    \end{proof}
    \begin{cor}
      Let $a,b$ be as above. A representation $R\in \mathbf{V}(e)$ satisfying $r_{u_i}\neq 0$ for all $i$ is unstable if and only if $r_a=r_b=0$.
    \end{cor}
    \begin{proof}
      Given a representation $R$ satisfying the above property, we can find in its orbit $R'$ under the  $\mathrm{PGL(\mathbf{1})}$ action such that $r'_{u_i}=1$ for all $i$. Since $R$ is stable if and only if $R'$ is, we assume $r_{u_i}=1$ for all $i$.\\ 
      In the last paragraph of proof of Theorem \ref{except locus}, we proved that the stable locus of $\mathbf{V}(e,u_1-1,\ldots,u_n-1)=\mathbb{A}^2$ is $\mathbb{A}^2\backslash\{(0,0)\}$ where the point $(0,0)$ corresponds to representation the unique representation in $\mathbf{V}(e,u_1-1,\ldots,u_n-1)$ with $r_a=r_b=0$.  Thus $R$ is unstable if and only if $r_a=r_b=0$.
    \end{proof}
    We now study the tangent space of points on $C$:
    \begin{prop}\label{smooth}
      If $R\in \mathbf{V}(e)$, then $\Ext_A^1(R,R)=\mathbf{k}^2$.
    \end{prop}
    \begin{proof}
      The theorem follows from multiple steps of reduction.
      Since the dimension of $\Ext^1$ only depends on equivalence classes of representations, we may assume $R$ to be in the standard form, i.e.,  $r_{u_i}=1$ for all $i$. Let
      \begin{equation*}
          0\to R\to \mathcal{E}\to R\to 0
      \end{equation*}
      be any extension.To give such an extension is equivalent to give for each vertex $i$, an extension
      \begin{equation}
          0\to R_i\xrightarrow{\iota_i} \mathcal{E}_i\xrightarrow{p_i} R_i\to 0
      \end{equation}
      and for each arrow $a$ from $i$ to $j$, a linear transformation $\epsilon_a:\mathcal{E}_i\to \mathcal{E}_j$ such that the following diagram is commutative:
      \[ \begin{tikzcd}
      0 \arrow{r} &R_i\arrow{r}{\iota_i}\ar{d}{r_a}& \mathcal{E}_i\arrow{r}{p_i}\ar{d}{\epsilon_a}&R_i\arrow{r}\ar{d}{r_a} &0 \\%
      0 \arrow{r} &R_j\arrow{r}{\iota_j}& \mathcal{E}_j\arrow{r}{p_j}&R_j\arrow{r} &0 
      \end{tikzcd}
      \]
      Note similar to the action of $GL(\vec{1})$ on representations, the group $GL(\vec{2})=(GL(2))^N$ where $N$ is the number of vertices in $Q$ acts on the set of self-extensions of $R$ in the following way:
      Let $M=(M_0,M_1,\ldots,M_n)\in   GL(\vec{2})$ by 
      $$M\cdot \epsilon_a=M_{t(a)}\epsilon_aM_{s(a)}^{-1}$$
      for any arrow $a$, and for all $i$, 
      \begin{align*}
          M\cdot \iota_i&=M_i\iota_i\\
          M\cdot p_i&=p_iM_i^{-1}
      \end{align*}
      Note all the operation are matrix multiplications on the right hand sides of the above three equations.
      Two self-extensions of $R$, $E$,$E'$ are isomorphic if they are in the same orbit of the $GL(\vec{2})$ action. It is standard that the isomorphism classes of self extensions form a vector space and we need to show it has dimension $2$.\\
      Let $\mathcal{E}$ be such an extension, first we perform a base change, i.e an action of some $M\in GL(\vec{2})$ such that 
      \begin{align}
          &\iota_l(x)=(x,0)\label{reduction 1}\\
          &p_l(x,y)=y\label{reduction 2}
      \end{align}
      for all $l\in\{0,1,\ldots,n\}$
      Using the commutative diagram above, we see now
      \[
      \epsilon_a=
      \begin{bmatrix}
      r_a& \lambda_a\\
      0&r_a
      \end{bmatrix}
      \]
      Hence the orbit of all self-extensions of $R$ under the action of $GL(\vec{2})$ is isomorphic to the action of self-extensions of $R$ satisfying equations (\ref{reduction 1}) and (\ref{reduction 2}) for all $l$, under the action of a subgroup $H\subset GL(\vec{2})$, where the $i$-th component of an element $M$ of $H$ has the form 
      \[
        M_i=
        \begin{bmatrix}
        1& c_i\\
        0&1
      \end{bmatrix}
      \]
      Notice all the matrix in our consideration now are upper-triagular matrices with identical diagonal elements. So the diagonal subgroup
      $$D=(M_1,M_1,\ldots,M_1)\subset H$$
      acts trivially.
      Second,
      for all $i<k$, we have $r_{u_i}=1$, thus 
      \[
        \epsilon_{u_i}=
        \begin{bmatrix}
        1& \lambda_{u_i}\\
        0&1
      \end{bmatrix}
      \]
    For any given
    \[
        M_k'=
        \begin{bmatrix}
        1&c_{k'}\\
        0&1
      \end{bmatrix}
      \]
      for each $i<k$ one can always choose an unique $M_i$ depending on $M_{k'}$ such that
    \begin{align*}
        \epsilon_{u_i}'&=M_{k'} \epsilon_{u_i} M_i^{-1}=
        \begin{bmatrix}
        1&0\\
        0&1
        \end{bmatrix}
    \end{align*}
    For any given
    \[
    M_k=
        \begin{bmatrix}
        1&c_{k}\\
        0&1
      \end{bmatrix}
      \]
    we can do the same thing to $u_j$ with $j>k$. Hence the orbit self-extensions of $R$ under the action of $GL(\vec{2})$ is the isomorphic to the set of extensions satisfying equations (\ref{reduction 1}) and (\ref{reduction 2}) and $\epsilon_{u_i}=I_2$ for all $i$, under the action of $k^2$, where $(a,b)$ acts on the extensions by setting
    \[
    M_k=
        \begin{bmatrix}
        1&a\\
        0&1
      \end{bmatrix} 
       M_{k'}=
        \begin{bmatrix}
        1&b\\
        0&1
      \end{bmatrix}
      \]
      and find $M_i$ for other $i$ as above.
      
    Now let $f,g$ be the two arrows in $Q$ that whose values gives the $\mathbb{A}^2$ in the above proof. Then either $r_f\neq 0$ or $r_g\neq 0$. Without loss of generality, we can let $r_f\neq 0$. Then one can adjust $M_k$ and $M_{k'}$ so that $e_f=I_2$. Noting the fact that the diagonal subgroup acts trivially, the set of isomorphism classes of self-extension of $R$ is isomorphic to the set of self-extensions of $R$ that satisfies equations (\ref{reduction 1}), (\ref{reduction 2}), $\epsilon_{u_i}=I_2$ for all $i$ and $\epsilon_f=r_fI_2$. Now we count the dimension of such extensions. We claim $\lambda_e$ and $\lambda_g$ determines the extension.\\ 
    If an arrow $c$ from $i$ to $j$ satisfies $j<k$. Then either $u_j\circ c=m_cu_i+e\circ (\sum_{l=1}^Ln_la_l)$ for some $m_c\neq 0$, $n_l\in \C$ and $a_l$ arrow from $i$ to $k$, or $u_j\circ c=e\circ (\sum_{l=1}^Ln_la_l)$ is a multiple of $e$. In the first case, we have 
    \begin{align*}
        \epsilon_c&=m_cI_2+ \epsilon_e(\sum_{l=1}^Ln_l\epsilon_{a_l})
        =\begin{bmatrix}
    m_c&\sum_{l=1}^L(n_l{r_{a_l}})\lambda_e\\
    0&m_c
    \end{bmatrix}
    \end{align*}
    Notice $\epsilon_c$ is determined by $\lambda_e$, $R$ and $Q$. In the second case, we have
    \begin{equation*}
        I_2 e_c=\epsilon_e(\sum_{l=1}^Ln_l\epsilon_{a_l})
    \end{equation*}
    Thus 
    \[
    \epsilon_c=
    \begin{bmatrix}
    0&\sum_{l=1}^L(n_l{r_{a_l}})\lambda_e\\
    0&0
    \end{bmatrix}
    \]
    Apply same argument to arrows from $i$ to $j$ with $i>k'$, we see all the linear transformations are determined by $\lambda_e$, $R$ and $Q$ except for arrows of the form i,ii and iii.
    
    Now any indecomposible arrow $h$ of  type i,ii,iii other than $f,g$ satisfies a linear relation involving $f,g$, determining $r_h$. We claim the linear relation also determines the $e_h$. We illustrate this using an example:\\
    Consider the case $\mathcal{TS}^{op}=\{E_1,H-E-E_1,E,H-E,H-E_1\}$, which is the second case of Part 2 in proof of Lemma \ref{key}. There are two arrows of type i, one is from $1$ to $3$, which we call $f_2$, another is from $2$ to $3$, which we call $f_1$.There are two arrows from $3'$ to $4$ of type ii. We denote them by $f_2$, $g_2$. We obtain
       \begin{align*}
        f_2\circ u_1&=u_4\circ f_1\circ e_1\\
        g_2\circ u_1&=u_4\circ g_1
    \end{align*}
    and concluded that the value of $f_2,g_2$ is determined by $f_1,g_1$. Using $f_1,g_1$ as $f,g$ in the previous paragraph and assume without loss of generality $r_f\neq 0$.
    we have
    \begin{equation*}
        \epsilon_{f_2}=r_f\epsilon_{e_1}\\
    \end{equation*}
    \begin{equation*}
        \epsilon_{g_2}=\epsilon_{g_1}\\
    \end{equation*}
    Note $e_1$ is from $1$ to $2<3$, hence the argument in previous paragraph applies. So $\epsilon_h$ is determined by $\lambda_e$ and $\lambda_{g_1}$.
    In conclusion, any isomorphism class of self-extension of $R$ contains a unique normal form, which is determined by $\lambda_e$ and $\lambda_g$. Thus $\Ext_A^1(R,R)=\mathbf{k}^2$.
    \end{proof}
    
    \begin{proof}[\textbf{Proof of Main Theorem and Theorem 1.6}]
      Set theoretically, $$M_\theta=U\cup C$$ where $U\cong S_0\backslash P$ by Corollary \ref{isom on open} and $C\cong\mathbb{P}^1$ by Theorem \ref{except locus}. Since $M_\theta$ is projective, it is connected and of dimension $2$. Again since $U\cong S_0\backslash P$, $M_\theta$ is smooth in the open set $U$. For any point $m\in C$, let $R$ be a representation in this isomorphism class, we notice by Proposition \ref{smooth}
      \begin{align*}
          \dim(T_mM_\theta)=\dim\Ext^1_A(R,R)=2
      \end{align*}
      thus $M_\theta$ is smooth along $C$ also. Thus $M_\theta$ is a smooth, connected projective scheme of dimension $2$, so $M_\theta$ is a surface. Moreover, since $M_\theta$ has a dense open set isomorphic to $S_0\backslash P$ whose complement is isomorphic to $\mathbb{P}^1$, it is the blow up of $S_0$ at $P$ and
      \begin{align*}
          f:M_\theta\to M_{\theta_0}
      \end{align*}
      is the blow up morphism.\ 
      Consider the diagram
      \[ \begin{tikzcd}
      S \arrow[r,dashed,"T"] \arrow{d}{\pi} & M_\theta \arrow{d}{f} \\%
      S_0 \arrow{r}{T_0}& M_{\theta_0}
      \end{tikzcd}
      \]
      and the composition
      \begin{align*}
          T_0\circ\pi:S\to M_{\theta_0}
      \end{align*}
      which is a birational morphism. $(T_0\circ \pi)^{-1}$ is defined except at the point $T_0(P)$. Hence by the universal property of blow up, there exists a birational morphism $T':S\to M_\theta$ so that 
      \[ \begin{tikzcd}
      S \arrow[r,"T'"] \arrow{d}{\pi} & M_\theta \arrow{d}{f} \\%
      S_0 \arrow{r}{T_0}& M_{\theta_0}
      \end{tikzcd}
      \]
      It is clear that $T'|_{S\backslash E}$ coincides with $T|_{S\backslash E}$. Since $S$ is a variety and $S\backslash E$ is a nonempty open subset, the rational map $T$ is defined on all of $S$, hence
      \begin{align}
          T: S\to M_\theta 
      \end{align}
      is a morphism. Now since both $S$ and $M_\theta$ are isomorphic to blow up of $S_0$ at $P$, and $T$ is birational, $T$ is an isomorphism.
    \end{proof}
    
    \begin{proof}[Proof of Corollary \ref{del pezzo}]
    Let $S$ be a smooth del Pezzo surface of degree $d\geq 3$. Suppose $S$ is not $\mathbb{P}^1\times\mathbb{P}^1$, then $S$ is obtained from blowing up $9-d$ points on $\mathbb{P}^2$ in general position (in any order). Thus by Main Theorem, it suffices to find a full strong exceptional collection of line bundles on $S$ obtaiend from standard agumentation from $\mathbb{P}^2$ satisfying (\ref{technical}), there are lots of them. See the Table 1 below for examples: It is easy to check that they are standard augmentations.\\
    
    For $S\cong \mathbb{P}^1\times\mathbb{P}^1$, the result follows from \cite{King}
    \end{proof}

\begin{table}[h]
\caption{Strong exceptional collection of line bundles on Del Pezzo surfaces satisfying (\ref{technical})}
\begin{center}
\begin{tabular}{|c|c|}
		 \hline
		 Degree$=K^2_S$	 &Strong exceptional toric systems
		 $E_{i\ldots k}:=\sum_{j=i}^{k}E_i$\\
		 \hline

		$8$& $H-E_1,E_1,H-E_1,H$ \\
		 
		 \hline	
	    $7$& $H-E_{12},E_2,E_1-E_2,H-E_1, H$ \\
		 
		 \hline
		$6$& $ H-E_{123},E_3,E_2-E_3,E_1-E_2,H-E_1,H$  \\
		
		 \hline	
		 $5$& $H-E_{123},E_3,E_2-E_3,E_1-E_2,H-E_{14},E_4.H-E_4$ \\
		 \hline
		$4$& $ H-E_{123},E_3,E_2-E_3,E_1-E_2,H-E_{145},E_5,E_4-E_5,H-E_4$ \\
		\hline	
		$3$&$E_4, E_2-E_4, H-E_{125}, E_5, E_5-E_1 H-E_{136},E_6,E_3-E_6,H-E_{234}$\\
			\hline	
\end{tabular}
\end{center}
\label{table_csadm}
\end{table}
\subsection*{Acknowledgments}  We would like to thank our advisor Valery Lunts for constant support and inspiring discussions. We wish to thank  Alexey Elagin, Chen Jiang, Junyan Xu, Wai-Kit Yeung, Sailun Zhan for helpful discussions. 
    
\end{document}